\documentclass[sn-mathphys]{sn-jnl}% Math and Physical Sciences Reference Style
\jyear{2021}%
\usepackage{amsmath}
\usepackage{amssymb}
\usepackage{amsthm}
\theoremstyle{thmstyleone}%
\newtheorem{theorem}{Theorem}[section]% meant for sectionwise numbers
\newtheorem{proposition}[theorem]{Proposition}%

\theoremstyle{thmstyletwo}%

\theoremstyle{thmstylethree}%
\newtheorem{definition}[theorem]{Definition}%
\newtheorem{lemma}[theorem]{Lemma}
\newtheorem{corollary}[theorem]{Corollary}
\newtheorem{example}[theorem]{Example}%
\newtheorem{remark}[theorem]{Remark}%
\numberwithin{equation}{section}
\begin{document}
 
\newpage
\title[The Cowen-Douglas Theory for Operator Tuples and Similarity]{The Cowen-Douglas Theory for Operator Tuples and Similarity}

\author{Kui Ji} \author{Shanshan Ji} \author{Hyun-Kyoung Kwon}\author{Jing Xu}



\abstract{We are concerned with the similarity problem for Cowen-Douglas operator tuples. The unitary equivalence counterpart was already investigated in the 1970's and geometric concepts including vector bundles and curvature appeared in the description. As the Cowen-Douglas conjecture show, the study of the similarity problem has not been so successful until quite recently. The latest results reveal the close correlation between complex geometry, the corona problem, and the similarity problem for single Cowen-Douglas operators. Without making use of the corona theorems that no longer hold in the multi-variable setting, we prove that the single operator results for similarity remain true for Cowen-Douglas operator tuples as well.}

\keywords{The Cowen-Douglas class, Similarity, Complex bundles, Curvature inequality}


\pacs[MSC]{Primary {47B13,} 32L05; Secondary {32A10,} {32A36,} 32Q05}

\maketitle

\section{Introduction}\label{sec1}
To study equivalence problems for bounded linear operators on Hilbert space to which standard methods do not apply, M. J. Cowen and R. G. Douglas introduced in the late 1970's, a class of operators with a holomorphic eigenvector bundle structure \cite{CD,CD2}. Prior to the introduction of the Cowen-Douglas class, the discussion on operator equivalence for even the adjoints of various shift operators (the most-mentioned entities in the study of operators) was non-existent. Their influential work connects concepts and results from complex geometry to the fundamental problem of determining operator equivalence. Note that when one considers operators that are defined on finite-dimensional Hilbert space, the well-known Jordan Canonical Form Theorem and the results by C. Pearcy \cite{Pearcy} and W. Specht \cite{Specht} give a complete answer to this problem.

For a complex separable Hilbert space $\mathcal{H}$, let $\mathcal{L}(\mathcal{H})$ denote the algebra of bounded linear operators on $\mathcal{H}$. For $m \in \mathbb{N}$, let $\mathbf{T}=(T_{1},\cdots,T_{m}) \in \mathcal{L}(\mathcal{H})^m$ and $\mathbf{S}=(S_{1},\cdots,S_{m})\in \mathcal{L}(\mathcal{H})^m$ be $m$-tuples of commuting operators on $\mathcal{H}$. If there is a unitary operator $U\in \mathcal{L}(\mathcal{H})$ such that $U\mathbf{T}=\mathbf{S}U$, then $\mathbf{T}$ and $\mathbf{S}$ are said to be \emph{unitarily equivalent} (denoted by $\mathbf{T}\sim_{u}\mathbf{S}$). If there is an invertible operator $X\in \mathcal{L}(\mathcal{H})$ such that $X\mathbf{T}=\mathbf{S}X$, then $\mathbf{T}$ and $\mathbf{S}$ are \emph{similar} (denoted by $\mathbf{T}\sim_{s}\mathbf{S}$).

Given $\mathbf{T}=(T_1, \cdots, T_m) \in \mathcal{L}(\mathcal{H})^m$, first define an operator $\mathscr{D}_{\mathbf{T}}: \mathcal{H}\longrightarrow\mathcal{H}\oplus\cdots\oplus\mathcal{H}$ by
$$\mathscr{D}_{\mathbf{T}}h=(T_{1}h,\cdots,T_{m}h),$$
for $h \in \mathcal{H}$. Let $\Omega$ be a bounded domain of the $m$-dimensional complex plane $\mathbb{C}^{m}$ and consider $w=(w_{1},\cdots,w_{m}) \in \Omega$. If one sets $\mathbf{T}-w:=(T_{1}-w_{1},\cdots,T_{m}-w_{m})$, then it is easily seen that  $\ker\mathscr{D}_{\mathbf{T}-w}=\bigcap \limits_{i=1}^{m}\ker(T_{i}-w_{i}).$

\begin{definition}
For $m, n \in \mathbb{N}$ and $\Omega \subset \mathbb{C}^m$, \emph{the Cowen-Douglas class} $\mathbf{\mathcal{B}}_{n}^{m}(\Omega)$ consists of $m$-tuples of commuting operators $\mathbf{T}=(T_{1},\cdots,T_{m}) \in \mathcal{L}(\mathcal{H})^m$ satisfying the following conditions:
\begin{itemize}
  \item [(1)]$\text{ran} \mathscr{D}_{\mathbf{T}-w}$ is closed for all $w \in \Omega$;
  \item [(2)]$\dim \ker\mathscr{D}_{\mathbf{T}-w}=n$ for all $w \in \Omega$; and
  \item [(3)] $\bigvee_{w \in \Omega} \ker\mathscr{D}_{\mathbf{T}-w}$ is dense in $\mathcal{H}$.
\end{itemize}
\end{definition}
For an $m$-tuple of commuting operators $\mathbf{T}=(T_{1},\cdots,T_{m}) \in \mathbf{\mathcal{B}}_{n}^{m}(\Omega),$ M. J. Cowen and R. G. Douglas proved in \cite{CD, CD2} that an associated holomorphic eigenvector bundle $\mathcal{E}_{\mathbf{T}}$ over $\Omega$ of rank $n$ exists, where
$$\mathcal{E}_\mathbf{T}=\{(w, x)\in \Omega\times \mathcal
H: x \in \ker\mathscr{D}_{\mathbf{T}-w}\},\quad \pi(w,x)=w.$$
Furthermore, it was shown that two operator tuples $\mathbf{T}$ and $\widetilde{\mathbf{T}}$ in $\mathbf{\mathcal{B}}_{n}^{m}(\Omega)$ are unitarily equivalent if and only if the vector bundles $\mathcal{E}_\mathbf{T}$ and $\mathcal{E}_{\widetilde{\mathbf{T}}}$ are equivalent as Hermitian holomorphic vector bundles. They also showed that every $m$-tuple $\mathbf{T}\in \mathcal{B}_{n}^{m}(\Omega)$ can be realized as the adjoint of an $m$-tuple of multiplication operators by the coordinate functions on a Hilbert space of holomorphic functions on $\Omega^{*}=\{\overline{w}: w \in \Omega\}$.

For $\mathbf{T} \in \mathbf{\mathcal{B}}_{n}^{m}(\Omega),$ if we let $\sigma=\{\sigma_{1},\ldots,\sigma_{n}\}$ be a holomorphic frame of $\mathcal{E}_\mathbf{T}$ and form the Gram matrix $h(w)=(\langle \sigma_j(w),\sigma_i(w) \rangle)_{i,j=1}^{n}$ for $w\in\Omega,$
then the \emph{curvature} $\mathcal{K}_\mathbf{T}$
and the corresponding \emph{curvature matrix} $K_{\mathbf{T}}$ with entries $K_{\mathbf{T}}^{i,j}, 1 \leq i, j \leq m$, of the bundle $\mathcal{E}_{\mathbf{T}}$ are given by the formulas
\begin{equation}\label{curvature}
\begin{split}
\mathcal{K}_\mathbf{T}(w)&=-\sum \limits_{i,j=1}^{m}\frac{\partial}{\partial \overline{w}_{j}}\left(h^{-1}(w)\frac{\partial h(w)}{\partial w_{i}} \right)dw_{i}\wedge d\overline{w}_{j},\\
&K_{\mathbf{T}}^{i,j}(w)=-\frac{\partial}{\partial \overline{w}_{j}}\left(h^{-1}(w)\frac{\partial h(w)}{\partial w_{i}} \right).
\end{split}
\end{equation}
Note that we omit the notation $\mathcal{E}$ in $\mathcal{K}_{\bold{T}}$ without any ambiguity. Since the above formulas depend on the selection of the holomorphic frame, they are also written as $\mathcal{K}_\mathbf{T}(\sigma)$ and $K_{\mathbf{T}}^{i,j}(\sigma)$, respectively, should the need arise.
In the special case of $\mathbf{T}\in \mathcal{B}_{1}^{m}(\Omega)$, the curvature of the line bundle $\mathcal{E}_{\mathbf{T}}$ can be defined alternately as
$$\mathcal{K}_\mathbf{T}(w)=-\sum \limits_{i,j=1}^{m}\frac{\partial^{2}\log\Vert\gamma(w)\Vert^{2}}{\partial w_{i}\partial \overline{w}_{j}}dw_{i}\wedge d\overline{w}_{j},$$
where $\gamma$ is a non-vanishing holomorphic section of $\mathcal{E}_\mathbf{T}$. In \cite{BKM}, S. Biswas, D. K. Keshari, and G. Misra showed that the curvature matrix $K_{\mathbf{T}}$ of any $\mathbf{T}\in\mathbf{\mathcal{B}}_{1}^{m}(\Omega)$ is negative-definite. Moreover, in \cite{CD}, it was shown that for $\mathbf{T},\widetilde{\mathbf{T}}\in \mathcal{B}_{1}^{1}(\Omega)$,  $\mathbf{T}\sim_{u}\widetilde{\mathbf{T}}$ if and only if $\mathcal{K}_\mathbf{T}=\mathcal{K}_{\widetilde{\mathbf{T}}}.$ The curvature $\mathcal{K}_{\bold{T}}$, along with certain covariant derivations of the curvature, form a complete set of unitary invariants for an operator  $\bold{T} \in \mathcal{B}_{n}^{1}(\Omega)$ and this is another main result of \cite{CD}. In \cite{CS},
R. E. Curto and N. Salinas established a relationship between the class $\mathcal{B}_{n}^{m}(\Omega)$ and generalized reproducing kernels to describe when two $m$-tuples are unitarily equivalent. A similarity result for Cowen-Douglas operators in geometric terms such as curvature had been much more difficult to obtain. In fact, the work of D. N. Clark and G. Misra in \cite{CM2, CM} showed that the Cowen-Douglas conjecture that similarity can be determined from the behavior of the quotient of the entries of curvature matrices was false.

The corona problem of complex analysis is closely related to operator theory and complex geometry \cite{Carleson1962, N, NKN, NF1970, U}. In particular, M. Uchiyama characterized the contractive operators that are similar to the adjoint of some multiplication operator in the work \cite{U} based on the corona theorem due to M. Rosenblum in \cite{R}. The following lemma given by N. K. Nikolski shows how to use the notion of projections $\mathcal{P}(w)$,  $\mathcal{P}(w)=\mathcal{P}(w)^2$, to solve the corona problem. The space $H^{\infty}_{E_{*}\rightarrow E}(\Omega)$ denotes the algebra of bounded analytic functions defined on a domain $\Omega \subset \mathbb{C}^m$ whose function values are bounded linear operators from a Hilbert space $E_{*}$ to another one $E$:

\begin{lemma}[Nikolski's Lemma]
Let $F\in H^\infty _{E_*\rightarrow E}(\Omega)$ satisfy $F^*(z)F(z)>\delta^2 $ for some $\delta>0$ and for all $z\in\Omega$.
Then $F$ is left invertible in $H^\infty _{E_*\rightarrow E}(\Omega)$ (i.e., there exists a $G\in H^\infty _{E\rightarrow E_*}(\Omega)$ such that
$GF\equiv I$) if and only if there exists a function $\mathcal{P}\in H^\infty _{E\rightarrow E}(\Omega)$ whose values are
projections (not necessarily orthogonal) onto $F(z)E$ for all $z\in\Omega$. Moreover, if such an analytic projection $\mathcal{P}$ exists, then one can find a left inverse
$G\in H^\infty _{E\rightarrow E_*}(\Omega)$ satisfying $\|G\|_{\infty}\leq \delta^{-1}\|\mathcal{P}\|_{\infty}$.
\end{lemma}
For $T\in \mathcal{B}_{n}^{1}(\mathbb{D})$, where $\mathbb{D} \subset \mathbb{C}$ denotes the unit disk, let $\Pi(w)$ denote the orthogonal projection onto $\ker(T-w)$, for each $w \in \mathbb{D}$.
In \cite{KT}, the third author and S. Treil described when a contraction
$T\in \mathcal{B}_{n}^{1}(\mathbb{D})$ is similar to $n$ copies of $M^*_z$, the adjoint of the shift operator on the
Hardy space of the unit disk $\mathbb{D}$, in terms of the curvature matrices. It was proven that
$T\sim{s}\bigoplus\limits_{1}^n M^*_z$ if and only if
$$\left \|\frac{\partial \Pi(w)}{\partial
w}\right\|^2_{\mathfrak{S}_2}-\frac{n}{(1-\vert w\vert^2)^2}\leq\frac{\partial^{2}\psi(w)}{\partial\overline{w}\partial w},$$ for some bounded subharmonic function $\psi$ defined on $\mathbb{D}$ and for all $w \in \mathbb{D}$. For $T\in \mathcal{B}_{n}^{1}(\mathbb{D})$,  $\text{trace }K_T=-\left\|\frac{\partial \Pi(w)}{\partial w}\right \|_{\mathfrak{S}_2}^2$ (see \cite{HJK, JS}), while $\text{trace }K_{M^*_z}=-\frac{1}{(1-\vert w\vert^2)^2}$ for the Hardy shift $M^*_z$. The result was then generalized to other shift operators in \cite{DKT}. Note that this is consistent with the curvature inequality result of G. Misra given in \cite{M}.

The similarity results mentioned above rely, to some extent, on a model theorem and a well-known $\overline{\partial}$-method that has been used extensively in recent years to solve numerous versions of the corona problem \cite{Lars1967}. For a contraction $T\in\mathcal{L}(\mathcal{H})$, the model theory of B. Sz.-Nagy and C. Foia\c{s} provides the canonical model as a complete set of invariants. Model theorems for operator tuples in $\mathcal{L}(\mathcal{H})^m$ were also studied under various assumptions \cite{NF1970, NF}. Since one cannot relate the similarity problem to the corona problem easily anymore in the multi-operator setting, we propose an alternative approach in managing the similarity of tuples of operators in $\mathcal{B}_{n}^{m}(\Omega)$ for $m>1$.

Inspired by the previous similarity results, we first give a sufficient condition for the similarity between $\mathbf{T} \in \mathcal{B}_{1}^{m}(\mathbb{B}_{m})$ and the adjoint of the operator tuple $\mathbf{M}_z=(M_{z_1}, M_{z_2}, \cdots, M_{z_m})$ on a weighted Bergman space defined on the unit ball $\mathbb{B}_m=\{z \in \mathbb{C}^m: \vert z\vert < 1 \}$ of $\mathbb{C}^m$. The characterization is given in terms of the defect operator $\mathcal{D}_{\mathbf{T}}$ corresponding to $\mathbf{T}$. Note that there already exist a number of necessary conditions for similarity -- it is a sufficient condition that had been hard to obtain. Throughout the paper,  $\mathbf{M}_z=(M_{z_1}, M_{z_2}, \cdots, M_{z_m})$ defined on a Hilbert space $\mathcal{H}$ of holomorphic functions on $\mathbb{B}_m$ will denote the tuple of multiplication operators by the coordinate functions
$$
(M_{z_i}f)(z)=z_if(z),
$$
for $f \in \mathcal{H}$ and $z \in \mathbb{B}_m$. It can be checked that  $\mathbf{M}_z^* \in \mathcal{B}_1^m(\mathbb{B}_m),$
where $\mathbf{M}^*_z=(M_{z_1}^*, M_{z_2}^*, \cdots, M_{z_m}^*)$.

\begin{theorem}
Let $\mathbf{T}=(T_{1},\cdots,T_{m})\in \mathcal{B}_{1}^{m}(\mathbb{B}_{m}) \subset \mathcal{L}(\mathcal{H})^m$ and consider the operator tuple $\mathbf{M}^*_z=(M^*_{z_1},\cdots,M^*_{z_m})$ on a weighted Bergman space $\mathcal{H}_k$, where $k > m+1$. Suppose that $(I-\sum\limits_{i=1}^{m}T_{i}^{*}T_{i})^{k}\geq0$  and $\lim\limits_{j}f_{j}(\mathbf{T}^{*}, \mathbf{T})h=0, h\in \mathcal{H}$, where $f_j(z,w)=\sum\limits_{i=j}^{\infty}\mathbf{e}_{i}(z)(1-\langle z, w\rangle)^{k}\mathbf{e}_{i}(w)^*$, for an orthonormal basis  $\{\mathbf{e}_i\}_{i=0}^{\infty}$ for $\mathcal{H}_{k}$.
If there exist a non-vanishing holomorphic section $t$ of $\mathcal{E}_\textbf{T}$ and a unit vector $\zeta_0 \in \overline{\text{ran } \mathcal{D}_{\mathbf{T}}}$ such that
\begin{equation*}
\sup\limits_{w\in\mathbb{B}_m}\frac{\Vert \mathcal{D}_{\mathbf{T}}t(w)\Vert^{2}}{\lvert\langle \mathcal{D}_{\mathbf{T}}t(w),\zeta_{0}\rangle\rvert^{2}}<\infty,
\end{equation*}
then $\mathbf{T} \sim_{s} \mathbf{M}_z^{*}$.
\end{theorem}

Although there is no general corona theorem that works for higher-order domains, a related condition that will be called \emph{condition $(\mathbf{C})$} in Section 3 will play a significant role in the formulation of another sufficient condition for similarity in terms of curvature matrices. The space $H^{\infty}(\Omega)$ denotes the collection of bounded analytic functions defined on the domain $\Omega$.

\begin{theorem}
Let $\mathbf{T}=(T_{1},\cdots,T_{m})\in \mathcal{B}_{n}^{m}(\mathbb{B}_m)$ and consider the operator tuple $\mathbf{M}_{z}^{*}=(M_{z_1}^{*},\cdots,M_{z_m}^{*}) \in \mathcal{B}_{n_1}^m(\mathbb{B}_m)$ on a Hilbert space $\mathcal{H}_{K}$ defined on $\mathbb{B}_m$ with reproducing kernel $K.$ Suppose that $n_2:=n/n_1 \in \mathbb{N}$ and that there exist an isometry $V$ and a Hermitian holomorphic vector bundle $\mathcal{E}$ over $\mathbb{B}_m$ such that
\begin{equation*}
VK_{\mathbf{T}, w^I{\overline{w}}^J}V^*-K_{\mathbf{M}_z^*,w^I{\overline{w}}^J} \otimes I_{n_2}=I_{n_1} \otimes {K}_{\mathcal{E},w^{I}\overline{w}^{J}}, \quad I,J \in\mathbb{N}_{0}^{m}.
\end{equation*}
If the bundle $\mathcal{E}$ satisfies condition $(\mathbf{C})$ via $\mathcal{H}_K$, then $\mathbf{T} \sim_{s} \bigoplus \limits_{1}^{n_2} \mathbf{M}_z^*$.
\end{theorem}

A number of corollaries are given in Section 4. The result of \cite{DKT} and \cite{KT} describing contractive Cowen-Douglas operators that are similar to ${M}_z^*$ using curvature is generalized in the commuting operator tuples setting. Moreover, this description is also used to obtain a sufficient condition for the similarity between arbitrary Cowen-Douglas operator tuples in $\mathcal{B}_1^m(\Omega)$. The space $\{\mathbf{S}\}'$ denotes the commutant of the operator tuple $\mathbf{S}$.

\begin{theorem}
Let $\mathbf{T}=(T_{1},\cdots,T_{m}),\mathbf{S}=(S_{1},\cdots,S_{m})\in \mathcal{B}_{1}^{m}(\Omega)$ be such that $\{\mathbf{S}\}'\cong H^{\infty}(\Omega)$. Suppose that
$$\mathcal{K}_{\mathbf{S}}(w)-\mathcal{K}_\mathbf{T}(w) = \sum\limits_{i,j=1}^{m}\frac{\partial ^{2}\psi(w)}{\partial w_{i}\partial\overline{w}_{j} }dw_{i}\wedge d\overline{w}_{j},\,\,\,\,w\in\Omega,$$
for some $\psi(w)=\log \sum\limits_{k=1}^{n}\lvert\phi_{k}(w)\rvert^{2}$, where $\phi_{k}$ are holomorphic functions defined on $\Omega$. If there exists an integer $l \leq n$ satisfying $\frac{\phi_k}{\phi_l} \in H^{\infty}(\Omega)$ for all $k \leq n$, then $\mathbf{T}\sim_{s}\mathbf{S}$, and ${K}_{\mathbf{T}}\leq {K}_\mathbf{S}$. In particular, when $m=1$, $\mathbf{T}$ and $\mathbf{S}$ are unitarily equivalent.
\end{theorem}

The inequalities and identities given in \cite{BKM}, \cite{KT}, and \cite{M} involving curvature matrices are extended as well in the final section.

\section{Preliminaries}\label{sec2}

\subsection{Reproducing kernel Hilbert spaces}\label{subsec2.1}
Let $\mathbb{B}_{m}=\{z\in\mathbb{C}^{m}:\lvert z\rvert<1\}$ be the open unit ball of $\mathbb{C}^{m}$. The space of all holomorphic functions defined on $\mathbb{B}_{m}$ will be denoted as $\mathcal{O}(\mathbb{B}_{m})$ while $H^{\infty}(\mathbb{B}_{m})$ will stand for the space of all bounded holomorphic functions on $\mathbb{B}_{m}$.
For a function $f\in\mathcal{O}(\mathbb{B}_{m})$, the radial derivative of $f$ is defined to be $Rf(z)=\sum\limits_{i=1}^{m}z_{i}\frac{\partial f}{\partial z_{i}}$. Once it is set that $R^{0}f(z)=f(z)$, we have for every $j \in \mathbb{N}$,
$R^{j}f(z)=R(R^{j-1}f(z))$. In particular, for a homogeneous polynomial $f$ of degree $n$, $Rf=nf.$ We will also need the familiar multi-index notation $\alpha=(\alpha_{1},\cdots,\alpha_{m})\in \mathbb{N}_{0}^{m}$. As is well-known, $\lvert\alpha\rvert=\lvert\alpha_{1}\rvert+\cdots+\lvert\alpha_{m}\rvert$ and $\alpha!=\alpha_{1}!\cdots\alpha_{m}!$. For $\alpha, \beta\in \mathbb{N}_{0}^{m}$, $\alpha+\beta=(\alpha_{1}+\beta_{1},\cdots,\alpha_{m}+\beta_{m})$ and $\alpha\leq\beta$ whenever $\alpha_{i}\leq\beta_{i}$ for every  $1\leq i\leq m.$

The work \cite{ARS, hartz, ZhaoZhu2008} offer good references for what follows. For a real number $k$, one can consider the family of holomorphic function spaces
$$\mathcal{H}_{k}=\left\{f=\sum\limits_{\alpha\in\mathbb{N}_{0}^{m}}a_{\alpha}z^{\alpha}\in\mathcal{O}(\mathbb{B}_{m})
:\sum\limits_{\alpha\in\mathbb{N}_{0}^{m}}\lvert a_{\alpha}\rvert^{2}\frac{\alpha!}{\lvert\alpha\rvert!}(\lvert\alpha\rvert+1)^{1-k}<\infty\right\}.$$ Recall that a reproducing kernel Hilbert space is a Hilbert space $\mathcal{H}$ of functions on a set $X$ with the property that the evaluation at each $x\in X$ is a bounded linear functional on $\mathcal{H}.$ By the Riesz representation theorem, for each $x\in X$, there exists a function $k_{x}\in\mathcal{H}$ such that for all $f \in \mathcal{H},$
$$\langle f, k_{x}\rangle=f(x).$$
The function $K:X\times X\longrightarrow\mathbb{C}$ defined by $K(x,y)=k_{y}(x)$ is called the \emph{reproducing kernel} of $\mathcal{H}.$
When $k>0$, $\mathcal{H}_k$ is a reproducing kernel Hilbert space with reproducing kernel
$$K(z,w)=\frac{1}{(1-\langle z, w\rangle)^{k}}.$$
If $n \in \mathbb{N}$ is such that $2n+k-m>0$, then the space $\mathcal{H}_k$ can also be represented as
\begin{equation}\label{5}
\mathcal{H}_{k}=\left\{f\in\mathcal{O}(\mathbb{B}_{m})
:\int_{\mathbb{B}_{m}}\lvert R^{n}f\rvert^{2}(1-\lvert z\rvert^{2})^{2n+k-m-1}dV(z)<\infty\right\},
\end{equation}
where $V$ denotes the normalized volume measure on $\mathbb{B}_{m}$. The spaces $\mathcal{H}_k$ are closely related to the analytic Besov-Sobolev spaces $\mathbf{B}_p^{\sigma}(\mathbb{B}_m)$. Recall that for $n \in \mathbb{N}_0$, $0\leq\sigma<\infty, 1<p<\infty,$ and $n+\sigma>m/p$, the space $\mathbf{B}_{p}^{\sigma}(\mathbb{B}_{m})$ contains $f\in\mathcal{O}(\mathbb{B}_{m})$ with $(\sum\limits_{\lvert\alpha\rvert<n}\lvert\frac{\partial^{\alpha}f}{\partial z^{\alpha}}(0)\rvert^{p}+\int_{\mathbb{B}_{m}}\lvert R^{n}f\rvert^{p}(1-\lvert z\rvert^{2})^{p(n+\sigma)}d\lambda_{m}(z))^{\frac{1}{p}}<\infty,$
where $dz$ denotes the Lebesgue measure on $\mathbb{C}^{m}$ and  $d\lambda_{m}(z)=(1-\lvert z\rvert^{2})^{-m-1}dz$ is the invariant measure on $\mathbb{B}_{m}$. Well-known examples in this family of spaces include
the Dirichlet space $B_2(\mathbb{B}_m)=\mathcal{H}_0=\mathbf{B}_{2}^{0}(\mathbb{B}_{m})$, the Drury-Arveson space $H_{m}^{2}=\mathcal{H}_1=\mathbf{B}_{2}^{1/2}(\mathbb{B}_{m})$, the Hardy space $H^{2}(\mathbb{B}_{m})=\mathcal{H}_m=\mathbf{B}_{2}^{m/2}(\mathbb{B}_{m})$, and the Bergman space $L^{2}(\mathbb{B}_{m})=\mathcal{H}_{m+1}=\mathbf{B}_{2}^{(m+1)/2}(\mathbb{B}_{m})$. Moreover, $\mathcal{H}_k$ with $0<k<1$ give weighted Dirichlet-type spaces while those with $k > m+1$ represent weighted Bergman spaces.

\subsection{Operator-valued multipliers}\label{subsec2.2}
As every Hilbert space $\mathcal{H}$ of functions on a set $X$ comes with a corresponding multiplier algebra
$$\text{Mult}(\mathcal{H})=\{f:X\rightarrow\mathbb{C} : fh\in\mathcal{H} \, \, \text{for all}\, \, h\in\mathcal{H}\},$$
it is natural to consider the multipliers of reproducing kernel Hilbert spaces.
For every multiplier $f\in\text{Mult}(\mathcal{H})$, there is an associated multiplication operator $M_f$ defined by $M_{f}h=fh$ with $\Vert f\Vert_{\text{Mult}(\mathcal{H})}=\Vert M_{f}\Vert$. In particular, for weighted Bergman spaces $\mathcal{H}_{k}$ with $k > m+1$, $\text{Mult}(\mathcal{H}_{k})=H^{\infty}(\mathbb{B}_{m})$. For $m\geq2$, the multiplier norm on the Drury-Arveson space $\mathcal{H}_{1}$ is no longer equal to the supremum norm on the unit ball and therefore, $\text{Mult}(\mathcal{H}_{1})\subsetneq H^{\infty}(\mathbb{B}_{m}).$

Let $E$ be a Hilbert space. The Hilbert space tensor product $\mathcal{H}_{k}\otimes E$ can be regarded as the space of all holomorphic functions $f:\mathbb{B}_{m}\rightarrow E$ with Taylor series $f(z)=\sum\limits_{\alpha\in\mathbb{N}_{0}^{m}}a_{\alpha}z^{\alpha}$, where $a_{\alpha} \in E$ and
$$\sum\limits_{\alpha\in\mathbb{N}_{0}^{m}}\Vert a_{\alpha}\Vert^{2}\frac{\alpha!}{\lvert\alpha\rvert!}(\lvert\alpha\rvert+1)^{1-k}<\infty.$$
Now, for Hilbert spaces $E_{1}$ and $E_{2}$, let $\Phi:\mathbb{B}_{m}\rightarrow \mathcal{L}(E_{1},E_{2})$ be an operator-valued function. Given $h\in\mathcal{H}_{k}\otimes E_{1}$,
we define a function $M_{\Phi}h: \mathbb{B}_{m}\rightarrow E_{2}$ as
$$M_{\Phi}h(z)=\Phi(z)h(z),\quad z\in\mathbb{B}_{m}.$$
Denote by $\text{Mult}(\mathcal{H}_{k}\otimes E_{1},\mathcal{H}_{k}\otimes E_{2})$ the space of all $\Phi$ for which $M_{\Phi}h\in\mathcal{H}_{k}\otimes E_{2}$ for every $h\in\mathcal{H}_{k}\otimes E_{1}$. An element $\Phi \in \text{Mult}(\mathcal{H}_{k}\otimes E_{1},\mathcal{H}_{k}\otimes E_{2})$ is said to be a \emph{multiplier} and $M_{\Phi}$ is called an \emph{operator of multiplication by $\Phi$}. The space $\text{Mult}(\mathcal{H}_{k}\otimes E_{1},\mathcal{H}_{k}\otimes E_{2})$ is endowed with the norm $\Vert\Phi\Vert=\Vert M_{\Phi}\Vert.$ We now list some basic properties of multipliers.

\begin{lemma}
For a weighted Bergman space $\mathcal{H}_{k}$ with $k > m+1$ and a Hilbert space $E$, $$\text{Mult}(\mathcal{H}_{k}\otimes E,\mathcal{H}_{k}\otimes \mathbb{C})=H_{E\rightarrow \mathbb{C}}^{\infty}(\mathbb{B}_{m}).$$
\end{lemma}
\begin{proof}
Taking $n=0$ in (\ref{5}), we have for every $F\in H_{E\rightarrow \mathbb{C}}^{\infty}(\mathbb{B}_{m})$ and $f(z)\otimes g\in\mathcal{H}_{k}\otimes E$,
\begin{eqnarray*}
\Vert M_{F}(f(z)\otimes g)\Vert_{\mathcal{H}_{k} \otimes \mathbb{C}}^{2}&=&\int_{\mathbb{B}_{m}}\lvert f(z)\otimes F(z)g\rvert^{2}(1-\lvert z\rvert^{2})^{k-m-1}dV(z)\\[4pt]\nonumber
&\leq&\Vert F\Vert_{\infty}^{2}\int_{\mathbb{B}_{m}}\lvert f(z)\rvert^{2} \Vert  g\Vert_{E}^{2}(1-\lvert z\rvert^{2})^{k-m-1}dV(z)\\[4pt]\nonumber
&=&\Vert F \Vert_{\infty}^{2}\Vert f(z)\otimes g\Vert_{\mathcal{H}_{k}\otimes E}^{2}.\nonumber
\end{eqnarray*}
This means that $\Vert M_{F}\Vert\leq\Vert F\Vert_{\infty},$ and therefore, $F\in\text{Mult}(\mathcal{H}_{k}\otimes E,\mathcal{H}_{k}\otimes \mathbb{C}).$
Conversely, since $\text{Mult}(\mathcal{H}_{k})\subset H^{\infty}(\mathbb{B}_{m})$,  $\text{Mult}(\mathcal{H}_{k}\otimes E,\mathcal{H}_{k}\otimes \mathbb{C})\subset H_{E\rightarrow \mathbb{C}}^{\infty}(\mathbb{B}_{m}).$
\end{proof}
The following results are well-known. The first lemma can be found in \cite{AM2002}:
\begin{lemma}\label{lem3.1}
Let $\Phi:\mathbb{B}_{m}\rightarrow \mathcal{L}(E_{1},E_{2})$ be an operator-valued function. If $\Phi\in\text{Mult}(\mathcal{H}_{K}\otimes E_{1},\mathcal{H}_{K}\otimes E_{2})$, then
$$M_{\Phi}^{*}(K(\cdot,\overline{w})\otimes f)=K(\cdot,\overline{w})\otimes \Phi(\overline{w})^{*}f,\quad w\in\mathbb{B}_{m}, f\in E_{2}.$$
Conversely, if $\Phi:\mathbb{B}_{m}\rightarrow \mathcal{L}(E_{1},E_{2})$ and the mapping
$K(\cdot,\overline{w})\otimes f\mapsto K(\cdot,\overline{w})\otimes \Phi(\overline{w})^{*}f$ extends to a bounded operator $X\in\mathcal{L}(\mathcal{H}_{K}\otimes E_{2},\mathcal{H}_{K}\otimes E_{1})$,
then $\Phi\in\text{Mult}(\mathcal{H}_{K}\otimes E_{1},\mathcal{H}_{K}\otimes E_{2})$ and $X=M_{\Phi}^{*}.$
\end{lemma}

\begin{lemma}\label{lem3.1.1.1}
Let $\mathbf{M}^*_{z,k,E_{i}}$ be the adjoint of the multiplication
tuple $(M_{z_{1}},\cdots,M_{z_{m}})$ on $\mathcal{H}_{k}\otimes E_{i}$, $i=1,2$.
If $\Phi\in\text{Mult}(\mathcal{H}_{k}\otimes E_{1},\mathcal{H}_{k}\otimes E_{2})$, then
$$M_{\Phi}^{*}\mathbf{M}^*_{z,k,E_{2}}=\mathbf{M}^*_{z,k,E_{1}}M_{\Phi}^{*}.$$
\end{lemma}

The following lemma, due to J. A. Ball, T. T. Trent, and V. Vinnikov, characterizes $\text{Mult}(\mathcal{H}_{1}\otimes E_{1},\mathcal{H}_{1}\otimes E_{2})$ for the Drury-Arveson space $\mathcal{H}_{1}$. For the proof and additional results, see \cite{BTV2001,EP2002}:

\begin{lemma}\label{lem3.2}
Let $\Phi:\mathbb{B}_{m}\rightarrow \mathcal{L}(\mathcal{H}_{1}\otimes E_{1},\mathcal{H}_{1}\otimes E_{2})$. Then the following statements are equivalent:
\begin{itemize}
  \item [(1)]$\Phi\in\text{Mult}(\mathcal{H}_{1}\otimes E_{1},\mathcal{H}_{1}\otimes E_{2})$ with $\Vert \Phi\Vert\leq 1.$
  \item [(2)]The kernel
  $$\mathfrak{ K}_{\Phi}(z,w)=\frac{I-\Phi(z)\Phi(w)^{*}}{1-\langle z,w\rangle},$$
  is a positive, sesqui-analytic, $\mathcal{L}(E_{2})$-valued kernel on $\mathbb{B}_{m}\times\mathbb{B}_{m}$, i.e., there is an auxiliary Hilbert space $\mathcal{H}$ and a holomorphic $\mathcal{L}(\mathcal{H}, E_{2})$-valued function $\Psi$ on $\mathbb{B}_{m}$ such that for all $z,w\in\mathbb{B}_{m},$
 $$\mathfrak{ K}_{\Phi}(z,w)=\Psi(z)\Psi(w)^{*}.$$
\end{itemize}
\end{lemma}

\subsection{Model theorem}\label{2.5}
Let $\mathbf{M}_z=(M_{z_1}, \cdots, M_{z_m})$ be the multiplication tuple on a  reproducing kernel Hilbert space $\mathcal{H}_K$ defined on $\mathbb{B}_m$ such that for every $1\leq i\leq m,$
$$(M_{z_i}f)(z)=z_i f(z),\quad f\in\mathcal{H}_K, z \in \mathbb{B}_m.$$ For an $m$-tuple of commuting operators $\mathbf{T}=(T_1, \cdots, T_m) \in \mathcal{L}(\mathcal{H})^m$ and a multi-index $\alpha=(\alpha_1, \cdots, \alpha_m) \in \mathbb{N}^m_0$, let $\mathbf{T}^{\alpha}=T_1^{\alpha_1}\cdots T_m^{\alpha_m}$ and $\mathbf{T}^*=(T_1^*, \cdots, T^*_m)$. Suppose that $1/K$ is a polynomial and that $\frac{1}{K}(\mathbf{T}^{*}, \mathbf{T})\geq0$, where given a polynomial $p(z, \omega)=\sum\limits_{I,J\leq \beta}\alpha_{I,J}z^{I}\omega^{J},\beta\in \mathbb{N}_{0}^m$, we let
$p(\mathbf{T}^{*}, \mathbf{T})=\sum\limits_{I,J\leq \beta}\alpha_{I,J}\mathbf{T}^{*I}\mathbf{T}^{J}$. The defect operator $\mathcal{D}_{\mathbf{T}}$ of $\mathbf{T}$ is then defined to be
$$
\mathcal{D}_{\mathbf{T}}=\frac{1}{K}(\mathbf{T}^*,\mathbf{T})^{\frac{1}{2}}.
$$

We next define a mapping $V: \mathcal{H}\rightarrow \mathcal{N}\subset\mathcal{H}_K\otimes \mathcal{H}$ as $$Vh=\sum\limits_i \mathbf{e}_i(\cdot)\otimes \mathcal{D}_{\mathbf{T}}\mathbf{e}_i(\mathbf{T}^{*})^{*}h,$$ for $h\in\mathcal{H}$, where $\mathcal{N}=\overline{ran\,V}$ and $\{\mathbf{e}_i\}_{i=0}^{\infty}$ is an orthonormal basis for $\mathcal{H}_K$.
Then according to the result of C. G. Ambrozie, M. Engli$\check{s}$, and V. M$\ddot{u}$ller in \cite{AEM}, $V$ is a unitary operator satisfying $VT_j=M_{z_j}^*V$ for $1\leq j\leq m$.
The study of a model theorem for bounded linear operators have been quite extensive and can be found in \cite{Agler1982, A2, AEM, Athavale1978, Athavale1992,CV1993,CV1995,MV1993,Pott1999}. The following model theorem for a tuple of commuting operators is stated in \cite{AEM}:

\begin{theorem}\label{11.19}
Consider the operator tuple $\mathbf{M}_z=(M_{z_1},\cdots,M_{z_m})$ on a Hilbert space $\mathcal{H}_K$ of holomorphic functions with reproducing kernel $K$ such that $1/K$ is a polynomial. For an orthonormal basis  $\{\mathbf{e}_i\}_{i=0}^\infty$ for $\mathcal{H}_K$, let $f_{j}(z,
w)=\sum\limits_{i=j}^{\infty}\mathbf{e}_{i}(z)\frac{1}{K}(z,w)\mathbf{e}_{i}(w)^*$. Then the following statements are equivalent:
\begin{itemize}
  \item [(1)] $\mathbf{T}=(T_1, \cdots, T_m) \in \mathcal{L}(\mathcal{H})^m$ is unitarily equivalent to the restriction of $\mathbf{M}_{z}^{*}$ to an invariant subspace.
  \item [(2)] $\frac{1}{K}(\mathbf{T}^{*}, \mathbf{T})\geq0$ and $\lim\limits_{j}f_{j}(\mathbf{T}^{*}, \mathbf{T})h=0$ for $h\in \mathcal{H}$.
\end{itemize}
\end{theorem}

\section{Similarity in the class $\mathcal{B}_{n}^{m}(\Omega)$}\label{sec3}

We first give a sufficient condition for the similarity between operator tuples $\mathbf{T} \in \mathcal{B}_{1}^{m}(\mathbb{B}_{m})$ and $\mathbf{M}_z^*$ on a weighted Bergman space by using the defect operator $\mathcal{D}_{\mathbf{T}}$ and the model theorem given previously. We then introduce condition $(\mathbf{C})$ for Hermitian holomorphic vector bundles (see Subsection 3.2) and use it together with the curvature and its covariant derivatives to characterize similarity in the class $\mathcal{B}_{n}^{m}(\Omega)$. The similarity of operators inside a specific subclass of $\mathcal{B}_{n}^{m}(\Omega)$ and the uniqueness of decomposition of Cowen-Douglas operators are then discussed.
\subsection{Model theorem and similarity}
We start by investigating the eigenvector bundle $\mathcal{E}_{\mathbf{T}}$ of $\mathbf{T} \in \mathcal{B}_1^m(\Omega)$.

\begin{lemma}\label{2112.14}
Let $\textbf{T}=(T_{1},\cdots,T_{m})\in \mathcal{B}^m_1(\Omega) \subset \mathcal{L}(\mathcal{H})^m$ and consider the operator tuple $\mathbf{M}_z=(M_{z_1},\cdots,M_{z_m})$ on a Hilbert space $\mathcal{H}_K$ of holomorphic functions with reproducing kernel $K$ such that $1/K$ is a polynomial. Suppose that $\mathbf{T}$ satisfies either one of the equivalent statements in Theorem \ref{11.19}. Then for any $t(w) \in \ker(\mathbf{T}-w),$ $$\|t(w)\|^2=K(\overline{w},\overline{w})\|\mathcal{D}_{\mathbf{T}}t(w)\|^2.$$

\end{lemma}
\begin{proof}
Let $\{\mathbf{e}_i\}_{i=0}^\infty$ be an orthonormal basis for $\mathcal{H}_K$.
Since $t(w)\in\ker(\mathbf{T}-w)$, $f(\mathbf{T})t(w)=f(w)t(w)$ for every $f \in \mathcal{O}(\Omega)$.
Defining a mapping $V: \mathcal{H}\rightarrow \mathcal{N}\subset\mathcal{H}_K\otimes \mathcal{H}$ as $$Vh=\sum\limits_i \mathbf{e}_i(\cdot)\otimes \mathcal{D}_{\mathbf{T}}\mathbf{e}_i(\mathbf{T}^{*})^*h,$$ for $h\in\mathcal{H}$ and $\mathcal{N}=\overline{\text{ran}\,V}$, we have
\begin{eqnarray*}
Vt(w)&=&\sum\limits_i \mathbf{e}_i(\cdot)\otimes \mathcal{D}_{\mathbf{T}}\mathbf{e}_i(\mathbf{T}^{*})^{*}t(w)
=\sum\limits_i \mathbf{e}_i(\cdot)\otimes \mathcal{D}_{\mathbf{T}}\mathbf{e}_i(\overline{w})^{*}t(w) \\ [4pt] \nonumber
&=&\sum\limits_i \mathbf{e}_i(\cdot)\mathbf{e}_i(\overline{w})^{*}\otimes \mathcal{D}_{\mathbf{T}}t(w)
=K(\cdot,\overline{w})\otimes \mathcal{D}_{\mathbf{T}}t(w). \nonumber
\end{eqnarray*}
From \cite{AEM}, $V$ is unitary and the result follows.
\end{proof}
The next theorem is the first of our main results of the paper.
\begin{theorem}\label{thm3.1}
Let $\mathbf{T}=(T_{1},\cdots,T_{m})\in \mathcal{B}_{1}^{m}(\mathbb{B}_{m}) \subset \mathcal{L}(\mathcal{H})^m$ and consider the operator tuple $\mathbf{M}_z^*=(M_{z_1}^*,\cdots,M_{z_m}^*)$ on a weighted Bergman space $\mathcal{H}_{k}$, where $k > m+1.$
Suppose that $(I-\sum\limits_{i=1}^{m}T_{i}^{*}T_{i})^{k}\geq0$  and $\lim\limits_{j}f_{j}(\mathbf{T}^{*}, \mathbf{T})h=0, h\in \mathcal{H}$, where $f_j(z,w)=\sum\limits_{i=j}^{\infty}\mathbf{e}_{i}(z)(1-\langle z, w\rangle)^{k}\mathbf{e}_{i}(w)^*$, for an orthonormal basis  $\{\mathbf{e}_i\}_{i=0}^{\infty}$ for $\mathcal{H}_{k}$.
If there exist a non-vanishing holomorphic section $t$ of $\mathcal{E}_\textbf{T}$ and a unit vector $\zeta_0 \in \overline{\text{ran } \mathcal{D}_{\mathbf{T}}}$ such that
\begin{equation}\label{4}
\sup\limits_{w\in\mathbb{B}_m}\frac{\Vert \mathcal{D}_{\mathbf{T}}t(w)\Vert^{2}}{\lvert\langle \mathcal{D}_{\mathbf{T}}t(w),\zeta_{0}\rangle\rvert^{2}}<\infty,
\end{equation}
then $\mathbf{T} \sim_{s} \mathbf{M}_z^{*}$.
\end{theorem}
\begin{proof}
Let $E=\overline{\text{ran } \mathcal{D}_{\mathbf{T}}}$ and note by Theorem \ref{11.19} that $\mathbf{T}\sim_{u}\mathbf{M}_{z,E}^*\vert_{\mathcal{N}}$, where $\mathcal{N}$ is an invariant subspace of $\mathbf{M}_{z,E}^*$. By Lemma \ref{2112.14}, we then have  $$\ker(\mathbf{M}_{z,E}^*\vert_{\mathcal{N}}-w)=\bigvee\{K(\cdot,\overline{w})\otimes\mathcal{D}_{\mathbf{T}}t(w):\mathcal{D}_{\mathbf{T}}t(w)\in E\}.$$
Given a unit vector $\zeta_{0} \in E$, one can select an orthonormal basis $\{\zeta_{\alpha}\}_{\alpha\geq0}$ of $E$ to express $\mathcal{D}_{\mathbf{T}}t(w)$ as $\mathcal{D}_{\mathbf{T}}t(w)=\sum\limits_{\alpha \geq 0}\langle \mathcal{D}_{\mathbf{T}}t(w),\zeta_{\alpha}\rangle\zeta_{\alpha}.$
If we set
$$\eta(w):=\mathcal{D}_{\mathbf{T}}t(w)-\langle \mathcal{D}_{\mathbf{T}}t(w),\zeta_{0}\rangle\zeta_{0} \in \mathcal{O}(\mathbb{B}_{m})\quad\text{and}\quad \psi(w):=\langle\mathcal{D}_{\mathbf{T}}t(w),\zeta_{0}\rangle\in \mathcal{O}(\mathbb{B}_{m}),$$
then $\Vert \mathcal{D}_{\mathbf{T}}t(w)\Vert^{2}=\Vert \eta(w)\Vert^{2}+\vert\psi(w)\vert^{2}$.

Now define an operator-valued function $F\in H^{\infty}_{E\rightarrow \mathbb{C}}(\mathbb{B}_{m})$ by
$$F^*(\overline{z})(\lambda):=\lambda\frac{ \eta(z)}{\psi(z)},\quad  \lambda\in\mathbb{C},z\in\mathbb{B}_{m}.$$
Since we know that for a weighted Bergman space $\mathcal{H}_{k}$ with $k > m+1,$
$$\text{Mult}(\mathcal{H}_{k}\otimes E,\mathcal{H}_{k}\otimes \mathbb{C})=H_{E\rightarrow \mathbb{C}}^{\infty}(\mathbb{B}_{m}),$$
Lemma \ref{lem3.1} and condition (\ref{4}) yield
$$M_{F}^{*}(K(\cdot,\overline{w})\otimes 1)=K(\cdot,\overline{w})\otimes F(\overline{w})^{*}(1)=K(\cdot,\overline{w})\otimes \frac{\eta(w)}{\psi(w)},\quad w\in\mathbb{B}_{m}.$$
Hence, $M_{F}^{*}$ is bounded and for $w \in \mathbb{B}_{m}$,
\begin{eqnarray*}
\Vert \mathcal{D}_{\mathbf{T}}t(w)\Vert^{2}&=&\Vert \eta(w)\Vert^{2}+\vert\psi(w)\vert^{2}\\[4pt]\nonumber
&=&\vert\psi(w)\vert^{2}\left(\Big{\Vert}\frac{ \eta(w)}{\psi(w)}\Big{\Vert}^{2}+1\right)\\[4pt]\nonumber
&=&\vert\psi(w)\vert^{2}\left(\frac{\Vert M_{F}^{*}(K(\cdot,\overline{w})\otimes 1)\Vert^{2}}{\Vert K(\cdot,\overline{w})\Vert^{2}}+1\right).\nonumber
\end{eqnarray*}
Since $\psi(w) \in \mathcal{O}(\mathbb{B}_{m})$, the definition of curvature from (\ref{curvature}) then gives
\begin{eqnarray*}
\mathcal{K}_{\mathbf{M}_{z,E}^*\vert_{\mathcal{N}}}(w)&=&-\sum^{m}_{i,j=1}\frac{\partial ^{2}\log \Vert K(\cdot,\overline{w})\otimes\mathcal{D}_{\mathbf{T}}t(w)\Vert^{2}}{\partial w_{i}\partial\overline{w}_{j}}dw_{i}\wedge d\overline{w}_{j}\\[4pt]\nonumber
&=&-\sum^{m}_{i,j=1}\frac{\partial ^{2}\log\Vert \mathcal{D}_{\mathbf{T}}t(w)\Vert^{2}}{\partial w_{i}\partial\overline{w}_{j}}dw_{i}\wedge d\overline{w}_{j}-\sum^{m}_{i,j=1}\frac{\partial ^{2}\log\Vert K(\cdot,\overline{w})\Vert^{2}}{\partial w_{i}\partial\overline{w}_{j}}dw_{i}\wedge d\overline{w}_{j}\\[4pt]\nonumber
&=&-\sum^{m}_{i,j=1}\frac{\partial ^{2}\log(\Vert M_{F}^{*}(K(\cdot,\overline{w})\otimes 1)\Vert^{2}+\Vert K(\cdot,\overline{w})\Vert^{2})}{\partial w_{i}\partial\overline{w}_{j}}dw_{i}\wedge d\overline{w}_{j}\\[4pt]\nonumber
&=&-\sum^{m}_{i,j=1}\frac{\partial ^{2}\log\langle(I+M_{F}M_{F}^{*})K(\cdot,\overline{w}),K(\cdot,\overline{w})\rangle}{\partial w_{i}\partial\overline{w}_{j}}dw_{i}\wedge d\overline{w}_{j}\\[4pt]\nonumber
&=&-\sum^{m}_{i,j=1}\frac{\partial ^{2}\log\Vert(I+M_{F}M_{F}^{*})^{\frac{1}{2}}K(\cdot,\overline{w})\Vert^{2}}{\partial w_{i}\partial\overline{w}_{j}}dw_{i}\wedge d\overline{w}_{j}.\nonumber
\end{eqnarray*}

Finally, let $$Y:=(I+M_{F}M_{F}^{*})^{\frac{1}{2}}.$$ Obviously,
$0\notin\sigma(Y)$ and
$YK(\cdot,\overline{w})\in \ker(Y\mathbf{M}_{z}^{*}Y^{-1}-w)$
so that for any $w\in\mathbb{B}_{m},$
$$\mathcal{K}_{\mathbf{M}_{z,E}^*\vert_{\mathcal{N}}}(w)=-\sum^{m}_{i,j=1}\frac{\partial ^{2}\log\Vert(I+M_{F}M_{F}^{*})^{\frac{1}{2}}K(\cdot,\overline{w})\Vert^{2}}{\partial w_{i}\partial\overline{w}_{j}}dw_{i}\wedge d\overline{w}_{j}=\mathcal{K}_{Y\mathbf{M}_{z}^{*}Y^{-1}}(w).$$
This shows that $\mathbf{M}_{z,E}^*\vert_{\mathcal{N}} \sim_{u} Y\mathbf{M}_{z}^{*}Y^{-1}$ and since $\mathbf{T}\sim_{u}\mathbf{M}_{z,E}^*\vert_{\mathcal{N}}$, $\mathbf{T}\sim_{s}\mathbf{M}_{z}^{*}$ as claimed.
\end{proof}

\begin{remark}
After an obvious modification of the condition $(I-\sum_{i=1}^m T^*_iT_i)^k \geq 0$ and the form of $f_j$ tailored to the reproducing kernel $K$, Theorem \ref{thm3.1} can be generalized to any operator tuple $\mathbf{M}_z^*$ on a reproducing kernel Hilbert space $\mathcal{H}_K$ such that $1/K$ is a polynomial as long as $\text{Mult}(\mathcal{H}_K)=H^{\infty}(\mathbb{B}_m)$. Moreover, one can use Lemma \ref{lem3.2} to check the multiplier algebra condition when working on the similarity between a row-contraction $\mathbf{T}\in \mathcal{B}_{1}^m(\mathbb{B}_m)$ and the operator tuple $\mathbf{M}_z^*=(M_{z_1}^*,\cdots,M_{z_m}^*)$ on the Drury-Arveson space $\mathcal{H}_{1}$.
\end{remark}

\subsection{Complex bundles and similarity}
Denote by $\{\sigma_{i}\}_{i=1}^{n}$ an orthonormal basis for $\mathbb{C}^{n}$ and let a Hilbert space $\mathcal{H}$ on $\Omega$ and  analytic vector valued functions $\{f_{i}\}_{i=1}^{n}$ over $\Omega$ be given, where $\Omega \subset \mathbb{C}^m$. Let $\mathcal{E}$ be an $n$-dimensional Hermitian holomorphic vector bundle over $\Omega$, $f_1,\ldots,f_n$ be $n$ holomorphic cross-sections of $\mathcal{E}$ which form a frame for $\mathcal{E}$ on $\Omega$.
For $w \in \Omega$, set $\mathcal{E}(w)=\bigvee\{f_1(w),\ldots, f_n(w)\}$ and $E=\bigvee_{w\in\Omega}\{f_1(w),\ldots, f_n(w)\}$. We will say that \emph{condition $(\mathbf{C})$ holds} for the Hermitian holomorphic vector bundle $\mathcal{E}$ via $\mathcal{H}$ if there exist functions $F\in H^{\infty}_{\mathbb{C}^{n}\rightarrow E}(\Omega)$ and $G\in H^{\infty}_{E\rightarrow\mathbb{C}^{n}}(\Omega)$ such that
$F^{\#}(\overline{w})(\sigma_{i}):=F(w)(\sigma_i)=f_{i}(w)$, $G^{\#}(\overline{w})(f_i(w)):=G(w)(f_i(w))$, $(F^{\#})^* \in \text{Mult}(\mathcal{H} \otimes E, \mathcal{H} \otimes \mathbb{C}^n)$, $(G^{\#})^* \in \text{Mult}(\mathcal{H} \otimes \mathbb{C}^n, \mathcal{H} \otimes E)$, and  $G^{\#}(\overline{w})F^{\#}(\overline{w})\equiv I$ for all $w\in\Omega.$ Using the curvature and covariant derivatives of complex bundles as well as condition $(\mathbf{C})$, we give a similarity description in the class $\mathcal{B}_{n}^{m}(\mathbb{B}_m)$.

\begin{lemma}\label{lem3.1.1}
Let $\mathcal{E}_1$ and $\mathcal{E}_2$ be Hermitian holomorphic bundles over $\Omega\subset\mathbb{C}^{m}$ of rank $n_1$ and of $n_2$, respectively. Then for any $I,J \in\mathbb{N}_{0}^{m}$,
$${K}_{\mathcal{E}_1\otimes \mathcal{E}_2, w^I\overline{w}^J}={K}_{\mathcal{E}_1, w^I\overline{w}^J}\otimes I_{n_2}+I_{n_1} \otimes {K}_{\mathcal{E}_2,w^I\overline{w}^J}.$$
\end{lemma}
\begin{proof}
Let $\{\phi_1, \phi_2, \ldots, \phi_{n_1}\}$ and $\{\gamma_{1},\gamma_{2},\ldots,\gamma_{n_2}\}$ be holomorphic frames of $\mathcal{E}_{1}$ and of $\mathcal{E}_{2}$, respectively. Then $\{\phi_1 \otimes\gamma_{1},\phi_1 \otimes\gamma_{2},\ldots,\phi_{n_1} \otimes \gamma_1, \ldots, \phi_{n_1}\otimes\gamma_{n_2}\}$ is a holomorphic frame of $\mathcal{E}_{1}\otimes \mathcal{E}_{2}$ and $h_{\mathcal{E}_{1}\otimes \mathcal{E}_{2}}=h_{\mathcal{E}_{1}} \otimes h_{\mathcal{E}_{2}}$ so that
\begin{eqnarray*}
&&{K}_{\mathcal{E}_{1}\otimes \mathcal{E}_{2}}\\[4pt]\nonumber
&=&\left(\frac{\partial}{\partial \overline{w}_{j}}\left(h_{\mathcal{E}_{1}\otimes \mathcal{E}_{2}}^{-1}\frac{\partial h_{\mathcal{E}_{1}\otimes \mathcal{E}_{2}}}{\partial w_{i}} \right)\right)_{i,j=1}^{m}\\[4pt]\nonumber
&=&\left(\frac{\partial}{\partial \overline{w}_{j}}\left[(h^{-1}_{\mathcal{E}_{1}}\otimes h^{-1}_{\mathcal{E}_{2}})\left(\frac{\partial h_{\mathcal{E}_{1}}}{\partial w_{i}}\otimes h_{\mathcal{E}_{2}}+h_{\mathcal{E}_{1}}\otimes\frac{\partial h_{\mathcal{E}_{2}}}{\partial w_{i}}\right)\right]\right)_{i,j=1}^{m}\\[4pt]\nonumber
&=&\left(\frac{\partial}{\partial \overline{w}_{j}}\left[(h^{-1}_{\mathcal{E}_{1}}\otimes h^{-1}_{\mathcal{E}_{2}})\left(\frac{\partial h_{\mathcal{E}_{1}}}{\partial w_{i}}\otimes h_{\mathcal{E}_{2}}\right)\right]\right)_{i,j=1}^{m}
+\left(\frac{\partial}{\partial \overline{w}_{j}}\left[(h^{-1}_{\mathcal{E}_{1}}\otimes h^{-1}_{\mathcal{E}_{2}})\left(h_{\mathcal{E}_{1}}\otimes\frac{\partial h_{\mathcal{E}_{2}}}{\partial w_{i}}\right)\right]\right)_{i,j=1}^{m}\\[4pt]\nonumber
&=&\left(\frac{\partial}{\partial \overline{w}_{j}}\left(h^{-1}_{\mathcal{E}_{1}}\frac{\partial h_{\mathcal{E}_{1}}}{\partial w_{i}}\otimes h^{-1}_{\mathcal{E}_{2}} h_{\mathcal{E}_{2}}\right)\right)_{i,j=1}^{m}
+\left(\frac{\partial}{\partial \overline{w}_{j}}\left(h^{-1}_{\mathcal{E}_{1}} h_{\mathcal{E}_{1}}\otimes
h^{-1}_{\mathcal{E}_{2}}\frac{\partial h_{\mathcal{E}_{2}}}{\partial w_{i}}\right)\right)_{i,j=1}^{m}\\[4pt]\nonumber
&=&{K}_{\mathcal{E}_1}\otimes I_{n_{2}}+I_{n_{1}}\otimes{K}_{\mathcal{E}_2}.
\end{eqnarray*}
It follows that for $e_{j}=(0,\cdots,1,\cdots,0)\in\mathbb{N}_{0}^{m}$ with $1$ in the $j$-th position,
\begin{eqnarray*}
{K}_{\mathcal{E}_1\otimes \mathcal{E}_2, w^{e_{j}}}&=&\frac{\partial {K}_{\mathcal{E}_1\otimes \mathcal{E}_2}}{\partial w_{j}}+\left[h_{\mathcal{E}_{1}\otimes \mathcal{E}_{2}}^{-1}\frac{\partial h_{\mathcal{E}_{1}\otimes \mathcal{E}_{2}}}{\partial w_{j}}, \,\,{K}_{\mathcal{E}_1\otimes \mathcal{E}_2}\right]\\[4pt]\nonumber
&=&\frac{\partial}{\partial w_{j}}\Big({K}_{\mathcal{E}_1}\otimes I_{n_{2}}+I_{n_{1}}\otimes {K}_{\mathcal{E}_2}\Big)
\\[4pt]\nonumber
&&\quad+\left[h^{-1}_{\mathcal{E}_{1}}\frac{\partial h_{\mathcal{E}_{1}}}{\partial w_{j}}\otimes I_{n_{2}}+I_{n_{1}}\otimes h^{-1}_{\mathcal{E}_{2}}\frac{\partial h_{\mathcal{E}_{2}}}{\partial w_{j}},\,\,{K}_{\mathcal{E}_1}\otimes I_{n_{2}}+I_{n_{1}}\otimes {K}_{\mathcal{E}_2}\right]\\[4pt]\nonumber
&=&\left[\frac{\partial{K}_{\mathcal{E}_1}}{\partial w_{j}}+h^{-1}_{\mathcal{E}_{1}}\frac{\partial h_{\mathcal{E}_{1}}}{\partial w_{j}}{K}_{\mathcal{E}_1}-{K}_{\mathcal{E}_1}h^{-1}_{\mathcal{E}_{1}}\frac{\partial h_{\mathcal{E}_{1}}}{\partial w_{j}}\right]\otimes I_{n_{2}}\\[4pt]\nonumber
&&\quad+I_{n_{1}}\otimes\left[\frac{\partial {K}_{\mathcal{E}_2}}{\partial w_{j}}+ h^{-1}_{\mathcal{E}_{2}}\frac{\partial h_{\mathcal{E}_{2}}}{\partial w_{j}}{K}_{\mathcal{E}_2}-{K}_{\mathcal{E}_2}h^{-1}_{\mathcal{E}_{2}}\frac{\partial h_{\mathcal{E}_{2}}}{\partial w_{j}}\right]\\[4pt]\nonumber
&=&{K}_{\mathcal{E}_1,w^{e_{j}}}\otimes I_{n_{2}}+I_{n_{1}}\otimes {K}_{\mathcal{E}_2,w^{e_{j}}},\nonumber
\end{eqnarray*}
and
$${K}_{\mathcal{E}_1\otimes \mathcal{E}_2, \overline{w}^{e_{j}}}=\frac{\partial {K}_{\mathcal{E}_1\otimes \mathcal{E}_2}}{\partial \overline{w}_{j}}
=\frac{\partial}{\partial \overline{w}_{j}}\Big({K}_{\mathcal{E}_1}\otimes I_{n_{2}}+I_{n_{1}}\otimes {K}_{\mathcal{E}_2}\Big)={K}_{\mathcal{E}_1, \overline{w}^{e_{j}}}\otimes I_{n_{2}}+I_{n_{1}}\otimes {K}_{\mathcal{E}_2, \overline{w}^{e_{j}}}.$$

Next, if
$${K}_{\mathcal{E}_1\otimes \mathcal{E}_2, w^J}={K}_{\mathcal{E}_1, w^J}\otimes I_{n_{2}}+I_{n_{1}}\otimes{K}_{\mathcal{E}_2,w^J}\,\,\text{and}\,\, {K}_{\mathcal{E}_1\otimes \mathcal{E}_2, \overline{w}^J}={K}_{\mathcal{E}_1, \overline{w}^J}\otimes I_{n_{2}}+I_{n_{1}}\otimes {K}_{\mathcal{E}_2,\overline{w}^J}$$ holds for some $J\in\mathbb{N}_{0}^{m}$,
then
\begin{eqnarray*}
{K}_{\mathcal{E}_1\otimes \mathcal{E}_2, w^{J+e_{j}}}&=&\frac{\partial {K}_{\mathcal{E}_1\otimes \mathcal{E}_2, w^{J}}}{\partial w_{j}}+\left[h_{\mathcal{E}_{1}\otimes \mathcal{E}_{2}}^{-1}\frac{\partial h_{\mathcal{E}_{1}\otimes \mathcal{E}_{2}}}{\partial w_{j}},\,\, {K}_{\mathcal{E}_1\otimes \mathcal{E}_2, w^{J}}\right]\\[4pt]\nonumber
&=&\frac{\partial}{\partial w_{j}}\Big ({K}_{\mathcal{E}_1, w^J}\otimes I_{n_{2}}+I_{n_{1}}\otimes {K}_{\mathcal{E}_2,w^J}\Big)\\[4pt]\nonumber
&&\quad+\left[h^{-1}_{\mathcal{E}_{1}}\frac{\partial h_{\mathcal{E}_{1}}}{\partial w_{j}}\otimes I_{n_{2}}+I_{n_{1}}\otimes h^{-1}_{\mathcal{E}_{2}}\frac{\partial h_{\mathcal{E}_{2}}}{\partial w_{j}},\,\,{K}_{\mathcal{E}_1, w^J}\otimes I_{n_{2}}+I_{n_{1}}\otimes {K}_{\mathcal{E}_2,w^J}\right]\\[4pt]\nonumber
&=&\left[\frac{\partial {K}_{\mathcal{E}_1, w^J}}{\partial w_{j}}+h^{-1}_{\mathcal{E}_{1}}\frac{\partial h_{\mathcal{E}_{1}}}{\partial w_{j}}{K}_{\mathcal{E}_1, w^J}-{K}_{\mathcal{E}_1, w^J}h^{-1}_{\mathcal{E}_{1}}\frac{\partial h_{\mathcal{E}_{1}}}{\partial w_{j}}\right]\otimes I_{n_{2}}\\[4pt]\nonumber
&&\quad+I_{n_{1}}\otimes\left[\frac{\partial {K}_{\mathcal{E}_2, w^J}}{\partial w_{j}}+h^{-1}_{\mathcal{E}_{2}}\frac{\partial h_{\mathcal{E}_{2}}}{\partial w_{j}}{K}_{\mathcal{E}_2, w^J}-{K}_{\mathcal{E}_2, w^J}h^{-1}_{\mathcal{E}_{2}}\frac{\partial h_{\mathcal{E}_{2}}}{\partial w_{j}}\right]\\[4pt]\nonumber
&=&{K}_{\mathcal{E}_1, w^{J+e_{j}}}\otimes I_{n_{2}}+I_{n_{1}}\otimes{K}_{\mathcal{E}_2, w^{J+e_{j}}},\nonumber
\end{eqnarray*}
and
\begin{eqnarray*}
{K}_{\mathcal{E}_1\otimes \mathcal{E}_2, \overline{w}^{J+e_{j}}}&=&\frac{\partial{K}_{\mathcal{E}_1\otimes \mathcal{E}_2, \overline{w}^J}}{\partial \overline{w}_{j}}\\[4pt]\nonumber
&=&\frac{\partial}{\partial \overline{w}_{j}}\left({K}_{\mathcal{E}_1, \overline{w}^J}\otimes I_{n_{2}}+I_{n_{1}}\otimes {K}_{\mathcal{E}_2,\overline{w}^J}\right)\\[4pt]\nonumber
&=&{K}_{\mathcal{E}_1, \overline{w}^{J+e_{j}}}\otimes I_{n_{2}}+I_{n_{1}}\otimes{K}_{\mathcal{E}_2,\overline{w}^{J+e_{j}}},\nonumber
\end{eqnarray*}
for any $e_{j}=(0,\cdots,1,\cdots,0)\in\mathbb{N}_{0}^{m}.$
Hence,
$${K}_{\mathcal{E}_1\otimes \mathcal{E}_2, w^J}={K}_{\mathcal{E}_1, w^J}\otimes I_{n_{2}}+I_{n_{1}}\otimes{K}_{\mathcal{E}_2,w^J}\,\,\text{and}\,\, {K}_{\mathcal{E}_1\otimes \mathcal{E}_2, \overline{w}^J}={K}_{\mathcal{E}_1, \overline{w}^J}\otimes I_{n_{2}}+I_{n_{1}}\otimes {K}_{\mathcal{E}_2,\overline{w}^J}$$ for all $J\in\mathbb{N}_{0}^{m}$.
Since without loss of generality, we have for $I, J-e_{j} \in\mathbb{N}_{0}^{m},$
$${K}_{\mathcal{E}_1\otimes \mathcal{E}_2, w^I\overline{w}^{J-e_{j}}}={K}_{\mathcal{E}_1, w^I\overline{w}^{J-e_{j}}}\otimes I_{n_{2}}+I_{n_{1}}\otimes{K}_{\mathcal{E}_2,w^I\overline{w}^{J-e_{j}}},$$
it is easy to see that
\begin{eqnarray*}
{K}_{\mathcal{E}_1\otimes \mathcal{E}_2, w^I\overline{w}^J}&=&\frac{\partial}{\partial \overline{w}_{j}}\left({K}_{\mathcal{E}_1, w^I\overline{w}^{J-e_{j}}}\otimes I_{n_{2}}+I_{n_{1}}\otimes{K}_{\mathcal{E}_2,w^I\overline{w}^{J-e_{j}}}\right)\\[4pt]\nonumber
&=&{K}_{\mathcal{E}_1, w^I\overline{w}^{J}}\otimes I_{n_{2}}+I_{n_{1}}\otimes{K}_{\mathcal{E}_2,w^I\overline{w}^{J}}.\nonumber
\end{eqnarray*}
\end{proof}
We are now ready to prove our second main theorem of the paper.

\begin{theorem}\label{thm3.2}
Let $\mathbf{T}=(T_{1},\cdots,T_{m})\in \mathcal{B}_{n}^{m}(\mathbb{B}_m)$ and consider the operator tuple $\mathbf{M}_{z}^{*}=(M_{z_1}^{*},\cdots,M_{z_m}^{*}) \in \mathcal{B}_{n_1}^m(\mathbb{B}_m)$ on a Hilbert space $\mathcal{H}_{K}$ defined on $\mathbb{B}_m$ with reproducing kernel $K.$ Suppose that $n_2:=n/n_1 \in \mathbb{N}$ and that there exist an isometry $V$ and a Hermitian holomorphic vector bundle $\mathcal{E}$ over $\mathbb{B}_m$ such that
\begin{equation}\label{07}
VK_{\mathbf{T}, w^I{\overline{w}}^J}V^*-K_{\mathbf{M}_z^*,w^I{\overline{w}}^J} \otimes I_{n_2}=I_{n_1} \otimes {K}_{\mathcal{E},w^{I}\overline{w}^{J}}, \quad I,J \in\mathbb{N}_{0}^{m}.
\end{equation}
If the bundle $\mathcal{E}$ satisfies condition $(\mathbf{C})$ via $\mathcal{H}_K$, then $\mathbf{T} \sim_{s} \bigoplus \limits_{1}^{n_2} \mathbf{M}_z^*$.
\end{theorem}

\begin{proof}
First, from $(\ref{07})$ and the definition of condition $(\textbf{C})$, we know that $\mathcal{E}$ is $n_{2}$-dimensional Hermitian holomorphic vector bundle. Without losing generality, assume that $\{f_{1},\ldots,f_{n_{2}}\}$ is a frame of $\mathcal{E}$.
It is easy to see that the Hermitian holomorphic vector bundle $\mathcal{E}_{\mathbf{M}^*_{z}} \otimes \mathcal{E}$ is expressed as
$$(\mathcal{E}_{\mathbf{M}_{z}^{*}}\otimes \mathcal{E})(w)=\bigvee\limits_{1 \leq i \leq n_1 \atop 1 \leq j \leq n_2}\{K(\cdot,\overline w)\widetilde{\sigma}_i \otimes f_{j}(w)\}, \quad w \in \mathbb{B}_m,$$
where $\{\widetilde{\sigma}_i\}_{i=1}^{n_1}$ is an orthonormal basis for $\mathbb{C}^{n_1}$. By Lemma \ref{lem3.1.1} and $(\ref{07})$, we then have for some isometry $V$ and for all $I, J \in\mathbb{N}_{0}^{m}$,
$${K}_{\mathcal{E}_{\mathbf{M}_{z}^{*}} \otimes \mathcal{E}, w^I{\overline{w}}^J}={K}_{\mathcal{E}_{\mathbf{M}_{z}^{*},} w^I{\overline{w}}^J}\otimes I_{n_2}+I_{n_1} \otimes {K}_{\mathcal{E},w^I{\overline{w}}^J}=V{K}_{\mathcal{E}_{\mathbf{T}},w^{I}\overline{w}^{J}}V^{*}.$$
It is also proven in \cite{CD3} that $\mathcal{E}_{\mathbf{M}_{z}^{*}} \otimes \mathcal{E}$ is congruent to $\mathcal{E}_{\mathbf{T}}$, that is, there is a unitary operator $U$ such that
\begin{equation}\label{08}
U(\mathcal{E}_{\mathbf{M}_{z}^{*}}\otimes \mathcal{E})(w)=\mathcal{E}_{\mathbf{T}}(w),\quad w\in\mathbb{B}_m.
\end{equation}

Next, since the Hermitian holomorphic vector bundle $\mathcal{E}$ satisfies condition $(\mathbf{C})$, for the holomorphic frame $\{f_{1},\ldots,f_{n_{2}}\}$ of $\mathcal{E}$, there exist $F\in H^{\infty}_{\mathbb{C}^{n_2}\rightarrow E}(\mathbb{B}_m)$ and $G \in H^{\infty}_{E \rightarrow \mathbb{C}^{n_2}}(\mathbb{B}_m)$ such that $G^{\#}(\overline{w})F^{\#}(\overline{w})=I$ for every $w \in \mathbb{B}_m$, where $E=\bigvee_{w\in\mathbb{B}_m}\{f_{1}(w),\ldots,f_{n_{2}}(w)\}.$
A proof similar to the one used in Lemma \ref{lem3.1} then implies that for $(F^{\#})^* \in \text{Mult}(\mathcal{H}_K \otimes E, \mathcal{H}_K \otimes \mathbb{C}^{n_2})$,
\begin{eqnarray*}
M^*_{(F^{\#})^*}(K(\cdot,\overline{w}) \widetilde{\sigma}_i \otimes \sigma_{j})&=&K(\cdot,\overline{w}) \widetilde{\sigma}_i \otimes F^{\#}(\overline{w})(\sigma_{j})\\[4pt]\nonumber
&=&K(\cdot,\overline{w})\widetilde{\sigma}_i \otimes F(w)(\sigma_{j})\\[4pt]\nonumber
&=&K(\cdot,\overline{w}) \widetilde{\sigma}_i \otimes f_{j}(w),\nonumber
\end{eqnarray*}
where $\{\sigma_j\}_{j=1}^{n_2}$ is an orthonormal basis for $\mathbb{C}^{n_2}$. In addition, for $f\in\mathcal{H}_{K}$ and $g\in E$,
\begin{eqnarray*}
M_{(F^{\#})^{*}}\left(f(z)\otimes g \right)&=&(F^{\# }(z))^{*}\left(\sum\limits_{i=0}^{\infty}\langle f(z), e_{i}\rangle e_{i}\otimes g \right)\\[4pt]\nonumber
&=&\sum\limits_{i=1}^{\infty}\langle f(z), e_{i}\rangle (F^{\#}(z))^{*}( e_{i}\otimes g)\\[4pt]\nonumber
&=&\sum\limits_{i=1}^{\infty}\langle f(z), e_{i}\rangle e_{i}\otimes  (F^{\#}(z))^{*}g\\[4pt]\nonumber
&=&f(z) \otimes  (F^{\#}(z))^{*}g,\nonumber
\end{eqnarray*}
where $\{e_{i}\}_{i=0}^{\infty}$ is an orthonormal basis for $\mathcal{H}_{K}$. Hence,
\begin{eqnarray*}
&&\Big \langle f(z)\otimes g, M^{*}_{(F^{\#})^{*}}(K(\cdot, \overline{w})\widetilde{\sigma}_{i}\otimes \sigma_{j})\Big \rangle \\[4pt]\nonumber
&=&\Big \langle M_{(F^{\#})^{*}}(f(z)\otimes g), K(\cdot, \overline{w})\widetilde{\sigma}_{i}\otimes \sigma_{j} \Big \rangle\\[4pt]\nonumber
&=&\Big \langle f(z)\otimes (F^{\#}(z))^{*}g, K(\cdot, \overline{w})\widetilde{\sigma}_{i}\otimes \sigma_{j}\Big \rangle\\[4pt]\nonumber
&=&\Big \langle f(z)\otimes \sum\limits_{k=1}^{n_{2}}\langle (F^{\#}(z))^{*}g, \sigma_{k}\rangle \sigma_{k}, K(\cdot, \overline{w})\widetilde{\sigma}_{i}\otimes \sigma_{j}\Big \rangle\\[4pt]\nonumber
&=&\sum\limits_{k=1}^{n_{2}}\Big \langle f(z)\langle (F^{\#}(z))^{*}g, \sigma_{k}\rangle\otimes \sigma_{k}, K(\cdot, \overline{w})\widetilde{\sigma}_{i}\otimes \sigma_{j}\Big \rangle\\[4pt]\nonumber
&=&\sum\limits_{k=1}^{n_{2}}\left\langle f(z)\langle (F^{\#}(z))^{*}g, \sigma_{k}\rangle,  K(\cdot, \overline{w})\widetilde{\sigma}_{i}\right\rangle \Big \langle \sigma_{k}, \sigma_{j} \Big \rangle\\[4pt]\nonumber
&=&\left\langle f({w})\langle (F^{\#}(\overline{w}))^{*}g, \sigma_{j}\rangle, \widetilde{\sigma}_{i}\right\rangle\\[4pt]\nonumber
&=&\langle f({w}), \widetilde{\sigma}_{i}\rangle\langle (F^{\#}(\overline{w}))^{*}g, \sigma_{j}\rangle\\[4pt]\nonumber
&=&\langle f(z),  K(\cdot, \overline{w})\widetilde{\sigma}_{i}\rangle\langle g, F^{\# }(\overline{w})\sigma_{j}\rangle\\[4pt]\nonumber
&=&\langle f(z)\otimes g,  K(\cdot, \overline{w})\widetilde{\sigma}_{i}\otimes F^{\# }(\overline{w})\sigma_{j}\rangle.\nonumber
\end{eqnarray*}
This shows that $M^{*}_{(F^{\#})^{ *}}(K(\cdot,\overline{w})\widetilde{\sigma}_{i}\otimes \sigma_{j})=K(\cdot,\overline{w})\widetilde{\sigma}_{i}\otimes F^{\#}(\overline{w})(\sigma_{j}).$

Furthermore, since $\mathcal{H}_{K}=\bigvee\limits_{w\in\mathbb{B}^m}\{K(\cdot,\overline{w})\xi:\xi\in\mathbb{C}^{n_{1}}\}$, we have
$\text{ran } M^{*}_{(F^{\#})^{*}}=\mathcal{H}_{K}\otimes E$
and $\Vert M^{*}_{(F^{\#})^{*}}\Vert=\Vert F^{\#} \Vert<\infty.$
Combining these results and taking into account the operator $M^{*}_{(G^{\#})^{*}}:\mathcal{H}_{K}\otimes E \longrightarrow \mathcal{H}_{K}\otimes \mathbb{C}^{n_{2}}$, we obtain
\begin{eqnarray*}
M^{*}_{(G^{\#})^{*}}M^{*}_{(F^{\#})^{*}}(K(\cdot,\overline{w})\widetilde{\sigma}_{i}\otimes \sigma_{j})&=&M^{*}_{(G^{\#})^{*}}\left(K(\cdot,\overline{w}\right)\widetilde{\sigma}_{i}\otimes F^{\#}(\overline{w})(\sigma_{j}))\\[4pt]\nonumber
&=&K(\cdot,\overline{w})\widetilde{\sigma}_{i}\otimes G^{\#}(\overline{w})F^{\#}(\overline{w})(\sigma_{j})\\[4pt]\nonumber
&=&K(\cdot,\overline{w})\widetilde{\sigma}_{i}\otimes \sigma_{j},\nonumber
\end{eqnarray*}
for $1\leq i\leq n_{1}$ and $1\leq j\leq n_{2}.$
Therefore, $M^{*}_{(G^{\#})^{*}}M^{*}_{(F^{\#})^{*}}\equiv I$ so that $M^{*}_{(F^{\#})^{*}}$ is an invertible operator satisfying
$$
M^{*}_{(F^{\#})^{*}}(\mathcal{E}_{\mathbf{\mathbf{M}_{z}^{*}}\otimes I_{n_{2}}}(w))=(\mathcal{E}_{\mathbf{M}_{z}^{*}}\otimes \mathcal{E})(w),\quad w\in\mathbb{B}^m.
$$
From this and $(\ref{08})$, we conclude that the invertible operator $UM^*_{(F^{\#})^*}$
establishes the similarity between $\mathbf{T}$ and $\bigoplus \limits_{1}^{n_2} \mathbf{M}_z^*$.
\end{proof}

The following corollary should be compared to M. Uchiyama's result in \cite{U}. It is already known that the result holds for weighted Bergman spaces $\mathcal{H}_k$, $k > m+1$:
\begin{corollary}
Let $\mathbf{T}=(T_{1},\cdots,T_{m})\in \mathcal{B}_{1}^{m}(\mathbb{B}_m)$ and consider the operator tuple $\mathbf{M}_{z}^{*}=(M_{z_1}^{*},\cdots,M_{z_m}^{*})$ on a Hilbert space $\mathcal{H}_{K}$ such that $\text{Mult}(\mathcal{H}_{K})=H^{\infty}(\mathbb{B}_m)$.
If $\mathcal{E}_{\mathbf{T}}= \mathcal{E}_{\mathbf{M}_{z}^{*}}\otimes \mathcal{E}$ for some Hermitian holomorphic vector bundle $\mathcal{E}$ over $\mathbb{B}_m$, where $\mathcal{E}(w)=\bigvee f(w)$ and $f$ is bounded by positive constants, then  $\mathbf{T} \sim_{s} \mathbf{M}_{z}^{*}.$
\end{corollary}

\begin{remark}
Condition $(\mathbf{C})$ is related to the corona theorem.
Let $\mathbf{T}=(T_{1},\cdots,T_{m})\in \mathcal{B}_{1}^{m}(\mathbb{B}_m)$ and consider the operator tuple $\mathbf{M}_{z}^{*}=(M_{z_1}^{*},\cdots,M_{z_m}^{*})$ on a Hilbert space $\mathcal{H}_{K}$ defined on $\mathbb{B}_m$. Suppose that $\mathcal{E}_{\mathbf{T}}= \mathcal{E}_{\mathbf{M}_{z}^{*}}\otimes \mathcal{E}$ for some Hermitian holomorphic vector bundle $\mathcal{E}$ over $\mathbb{B}_m$, where $\mathcal{E}(w)=\bigvee f(w)$ is such that $f(w)=(f_{1}(w),f_{2}(w),\cdots,f_{n}(w))^{T}$ and  $f_j \in \text{Mult}(\mathcal{H}_K), 1 \leq j \leq n$. Let there exist a constant $\delta >0$ satisfying
\begin{equation}\label{Bezout}
\delta \leq \left(\sum_{j=1}^n \vert f_j(w)\vert^2 \right)^{\frac{1}{2}}, \quad w \in \mathbb{B}_m,
\end{equation}
and a function $F \in H^{\infty}_{\mathbb{C} \rightarrow E}(\mathbb{B}_m)$ such that
$$F^{\#}(\overline{w})(\lambda):=F(w)(\lambda)(w)={\lambda}f(w) \text{ and }(F^{\#})^* \in \text{Mult}(\mathcal{H}_K \otimes E, \mathcal{H}_K \otimes \mathbb{C}),
$$
for $\lambda \in \mathbb{C}$ and $E=\bigvee\limits_{w\in\mathbb{B}_m}f(w)$. Then condition (\ref{Bezout}) implies the existence of $g_1, \cdots g_n \in \text{Mult}(\mathcal{H}_K)$
such that
$$
\sum_{j=1}^n g_j(w)f_j(w)=1, \quad w \in \mathbb{B}_m.
$$
Setting $G=(g_1, g_2, \cdots, g_n)$, one can proceed as in the proof of Theorem \ref{thm3.2} to conclude that $\mathbf{T} \sim_{s} \mathbf{M}^*_z$.
\end{remark}

\begin{corollary}
Let $T\in \mathcal{B}_{n}^{1
}(\mathbb{D})$ and $S=(M_{z}^{*}, \mathcal{H}_{K}) \in \mathcal{B}_{n_1}^1(\mathbb{D})$. If
$$\text{Mult}(\mathcal{H}_{K})= H^{\infty}(\mathbb{D})\quad\text{and}\quad \mathcal{E}_{T}\sim_{u}\mathcal{E}_{S}\otimes \mathcal{E},$$
for some Hermitian holomorphic vector bundle $\mathcal{E}$ of rank $n_2:=n/{n_1}$, then $T\sim_{s}\bigoplus\limits_{i=1}^{n_2}S$ if and only if $\mathcal{E}$ satisfies condition $(\mathbf{C})$ via $\mathcal{H}_K$.
\end{corollary}
\begin{proof}
Suppose first that there exists a bounded invertible operator $X$ such that
$X(\bigoplus\limits_{i=1}^{n_2}S)=TX.$
Then
$X:\mathcal{E}_{S}\otimes \mathbb{C}^{n_2}\longrightarrow \mathcal{E}_{S}\otimes \mathcal{E}$
satisfies
$$
X(K(\cdot,\overline{w}) \widetilde{\sigma}_i \otimes \sigma_{j})=K(\cdot,\overline{w}) \widetilde{\sigma}_i \otimes f_{j}(w),
$$
for a holomorphic frame $\{f_{1},\cdots,f_{n_2}\}$ of $\mathcal{E}$ and orthonormal bases $\{\widetilde{\sigma}_j\}_{i=1}^{n_1}$ and $\{\sigma_{j}\}_{j=1}^{n_2}$ for $\mathbb{C}^{n_1}$ and for $\mathbb{C}^{n_2}$, respectively. Note that since $X$ is a bounded linear operator, the $f_{i}$ are uniformly bounded on $\mathbb{D}.$
Thus, we can define a function $F\in H^{\infty}_{\mathbb{C}^{n_2}\rightarrow E}(\mathbb{D})$ as
$$
F(w)\sigma_{j}=f_{j}(w),\quad 1\leq j \leq n_2,
$$
where $E=\bigvee\limits_{w\in\mathbb{D}}\{f_{j}(w):1\leq j \leq n_2\}.$
Obviously, the function $F^{\#}$ defined on $\mathbb{D}$ as $F^{\#}(w):=F(\overline{w})$ is such that
$
(F^{\#})^ * \in H^{\infty}_{E \rightarrow\mathbb{C}^{n_2}}(\mathbb{D}).
$
Moreover, since
$\text{Mult}(\mathcal{H}_{K}\otimes E,\mathcal{H}_{K}\otimes \mathbb{C}^{n_2})= H^{\infty}_{E \rightarrow\mathbb{C}^{n_2}}(\mathbb{D}),
$
$$(F^{\#})^ *\in \text{Mult}(\mathcal{H}_{K}\otimes E,\mathcal{H}_{K}\otimes \mathbb{C}^{n_2}).$$
Next, note from Theorem
\ref{thm3.2} that
$$\begin{array}{llll}
M^*_{(F^{\#})^*}(K(\cdot,\overline{w})\widetilde{\sigma_i}\otimes \sigma_{j})&=&
K(\cdot,\overline{w})\widetilde{\sigma}_i \otimes F^{\#}(\overline{w})\sigma_{j}\\
&=&K(\cdot,\overline{w})\widetilde{\sigma}_i \otimes F(w)\sigma_{j}\\
&=&K(\cdot,\overline{w}) \widetilde{\sigma}_i \otimes f_{j}(w),
\end{array}$$
and that the operator $M^*_{(F^{\#})^*}$ has dense range.
Since $X$ is invertible, this means that
$X=M^*_{(F^{\#})^*}$. Furthermore,
for every $h=K(\cdot,\overline{w})\xi\in\mathcal{H}_{K}$, $g \in \mathbb{C}^{n_2}$ and $w\in\mathbb{D},$ there exists a $\delta> 0$ such that
$$\begin{array}{llll}
\langle X^*X (h \otimes g), h \otimes g \rangle&=&
\Big\langle M_{(F^{\#})^*}M^*_{(F^{\#})^*}(h \otimes g), h \otimes g \Big \rangle\\
&=&\|h\|^2 \langle F^*(w)F(w)g, g \rangle\\
&\geq&\delta ^2 \|g\|^2\|h\|^2.
\end{array}$$
It follows that since $F\in H^{\infty}_{\mathbb{C}^{n_2}\rightarrow E}(\mathbb{D})$ and $F^*(w)F(w)\geq \delta^{2}>0$, there exists a function $G\in H^{\infty}_{E\rightarrow \mathbb{C}^{n_2}}(\mathbb{D})$ such that
$G(w)F(w)\equiv I$, that is,  $G^{\#}(\overline{w})F^{\#}(\overline{w})=I$, for every $w \in \mathbb{D}$.

Conversely, if the complex bundle $\mathcal{E}$ satisfies condition $(\mathbf{C})$, then there is a holomorphic frame $\{f_{1},\cdots,f_{n_2}\}$ of $\mathcal{E}$ and functions $F\in H^{\infty}_{\mathbb{C}^{n_2}\rightarrow E}(\mathbb{D})$ and $G\in H^{\infty}_{E\rightarrow\mathbb{C}^{n_2}}(\mathbb{D})$ such that
$$
F(w)\sigma_{i}=f_{i}(w) \quad \text{and} \quad G^{\#}(\overline{w})F^{\#}(\overline{w})\equiv I,
$$
for all $w \in \mathbb{D}$, $1 \leq i \leq n_2$, and an orthonormal basis $\{\sigma_1, \cdots, \sigma_{n_2}\}$ of $\mathbb{C}^{n_2}$.
Another application of Theorem \ref{thm3.2} yields a bounded invertible operator $M^*_{(F^{\#})^*}: \left(\mathcal{E}_{\oplus_{i=1}^{n_2} S}\right)(w) \rightarrow (\mathcal{E}_S \otimes \mathcal{E})(w)$. Now, since $\mathcal{E}_{T}\sim_{u}\mathcal{E}_{S}\otimes \mathcal{E}$, there is a unitary operator $U$ such that
$U\Big(\big(\mathcal{E}_{S}\otimes \mathcal{E}\big)(w)\Big)=\mathcal{E}_{T}(w)$ for $w \in \mathbb{D}.$
Then, $UM^*_{(F^{\#})^*}$ is a bounded invertible operator establishing the similarity between
$T$ and $\bigoplus\limits_{i=1}^{n_2}S$.
\end{proof}

The Dirichlet space $\mathcal{D}$ consists of all analytic functions $f(z)=\sum\limits_{n=0}^{\infty}a_{n}z^{n}$ defined on the unit disk $\mathbb{D} \subset \mathbb{C}$ satisfying
$\Vert f\Vert=\sum\limits_{n=0}^{\infty}(n+1)\vert a_{n} \vert^{2}<\infty.$
It is well-known that the reproducing kernel of  $\mathcal{D}$ is given as $K(z,w)=\frac{1}{\overline{w}z}\log\frac{1}{1-\overline{w}z}$, for $w, z \in \mathbb{D}$. Research on similarity on the Dirichlet space $\mathcal{D}$ can be found in \cite{HKK2016}, for instance. Using the results of \cite{Luo2022}, we now give a sufficient condition for a Cowen-Douglas operator to be similar to $M_z^*$ on $\mathcal{D}$.

\begin{corollary}
Let $T\in \mathcal{B}_{1}^{1}(\mathbb{D})$ and consider the operator $M_{z}^{*}\in \mathcal{B}_{1}^{1}(\mathbb{D})$ on the Dirichlet space $\mathcal{D}$.
Suppose that $\mathcal{E}_{T}= \mathcal{E}_{M_{z}^{*}}\otimes \mathcal{E}$ for some Hermitian holomorphic vector bundle $\mathcal{E}$ over $\mathbb{D}$, where $\mathcal{E}(w)=\bigvee f(w)$ and
$f(w)=(f_{1}(w),\cdots,f_n(w))^{T}$ for $f_{j}\in H^{\infty}(\mathbb{D})$. If
$$\int_{\mathbb{D}}\vert f(z)\vert^{2}\vert f_{j}^{'}(z)\vert^{2}U(z)dA(z)\lesssim\Vert f\Vert^{2}_{\mathcal{D}},\quad f\in\mathcal{D},$$
where
$$U(z)=\displaystyle\int_{\mathbb{T}}\log\left\vert\frac{1-\overline{w}z}{z-w}\right\vert^{2}\frac{\vert dw\vert}{2\pi(1-\vert w\vert^{2})}+\int_{\mathbb{T}}\frac{1-\vert z\vert^{2}}{\vert1-\overline{w}z\vert^{2}}\frac{\vert dw\vert}{2\pi}, \quad z \in \mathbb{D},$$
and $dA$ is the normalized area measure, then $T \sim_{s} M_{z}^{*}.$
\end{corollary}

In order to investigate similarity between tuples of irreducible operators in the Cowen-Douglas class $\mathcal{B}_{n}^{m}(\mathbb{B}_{m})$, we introduce a subclass denoted $\mathcal{F}\mathcal{B}_{n}^{m}(\mathbb{B}_{m})$. The subclass $\mathcal{F}\mathcal{B}_{n}^{m}(\mathbb{B}_{m})$ is the collection of all $\mathbf{T}=(T_{1},\cdots,T_{m})\in \mathcal{B}_{n}^{m}(\mathbb{B}_{m})$ of the form
 $$\mathbf{T}=\left( \begin{scriptsize}\begin{pmatrix}
T_{1,1}&\widetilde{T}_{1,2}&\cdots&\widetilde{T}_{1,n-1}&\widetilde{T}_{1,n} \\
0&T_{1,2}&\cdots&\widetilde{T}_{2,n-1}&\widetilde{T}_{2,n} \\
\vdots&\vdots&\ddots&\vdots&\vdots\\
0&0&\cdots&T_{1,n-1}&\widetilde{T}_{n-1,n}\\
0&0&\cdots&0&T_{1,n}\\
\end{pmatrix}\end{scriptsize},\cdots,\begin{scriptsize}\begin{pmatrix}
T_{m,1}&\widetilde{T}_{1,2}&\cdots&\widetilde{T}_{1,n-1}&\widetilde{T}_{1,n} \\
0&T_{m,2}&\cdots&\widetilde{T}_{2,n-1}&\widetilde{T}_{2,n} \\
\vdots&\vdots&\ddots&\vdots&\vdots\\
0&0&\cdots&T_{m,n-1}&\widetilde{T}_{n-1,n}\\
0&0&\cdots&0&T_{m,n}\\
\end{pmatrix}\end{scriptsize}\right),$$
where $\mathbf{T}_{i}=(T_{1,i},\cdots,T_{m,i})\in\mathcal{B}_{1}^{m}(\mathbb{B}_{m})$ and the $\widetilde{T}_{i,i+1}$, $1\leq i\leq n-1$, are non-zero operators such that $T_{k,i}\widetilde{T}_{i,i+1}=\widetilde{T}_{i,i+1}T_{k,i+1}$ for all $1\leq k\leq m$. The class $\mathcal{FB}_{n}^{1}(\Omega)$ was defined in \cite{JJDG}, and in \cite{JKSX,JJD}, the corresponding similarity question was considered.
 We now give a sufficient condition for the similarity between $\mathbf{T}$ and $\mathbf{S}$, where
$$\mathbf{S}=\left(\begin{scriptsize}\begin{pmatrix}
S_{1,1}&\widetilde{S}_{1,2}&\cdots&\widetilde{S}_{1,n-1}&\widetilde{S}_{1,n} \\
0&S_{1,2}&\cdots&\widetilde{S}_{2,n-1}&\widetilde{S}_{2,n} \\
\vdots&\vdots&\ddots&\vdots&\vdots\\
0&0&\cdots&S_{1,n-1}&\widetilde{S}_{n-1,n}\\
0&0&\cdots&0&S_{1,n}\\
\end{pmatrix}\end{scriptsize},\cdots,\begin{scriptsize}\begin{pmatrix}
S_{m,1}&\widetilde{S}_{1,2}&\cdots&\widetilde{S}_{1,n-1}&\widetilde{S}_{1,n} \\
0&S_{m,2}&\cdots&\widetilde{S}_{2,n-1}&\widetilde{S}_{2,n} \\
\vdots&\vdots&\ddots&\vdots&\vdots\\
0&0&\cdots&S_{m,n-1}&\widetilde{S}_{n-1,n}\\
0&0&\cdots&0&S_{m,n}\\
\end{pmatrix}\end{scriptsize}\right)$$
is in $\mathcal{F}\mathcal{B}_n^m(\mathbb{B}_{m})$ and $\mathbf{S}_{i}=(S_{1,i},\cdots,S_{m,i})$ is in $\mathcal{B}^{m}_1(\mathbb{B}_{m})$
for $1 \leq i \leq n$.

\begin{theorem}\label{thm3.4}
Let $\mathbf{T}, \mathbf{S} \in \mathcal{FB}^{m}_n(\mathbb{B}_{m})$, where $\mathbf{S}_{i} \sim_{u} ({\mathbf{M}}^*_{z}, \mathcal{H}_{k_i})$ for some $k_{i} >  m+1$ and for all $1\leq i\leq n$.
Suppose that the following conditions hold:
\begin{enumerate}
\item[(1)]$\mathcal{K}_{\mathbf{S}_{1}}(w)-\mathcal{K}_{\mathbf{T}_{1}}(w)=\sum \limits_{i,j=1}^{m}\frac{\partial^{2}\psi(w)}{\partial w_{i}\partial \overline{w}_{j}}dw_{i}\wedge d\overline{w}_{j},$
where $\psi(w)=\log \Vert f(w)\Vert ^{2}$ for some analytic vector valued function $f$ over $\mathbb{B}_m$.
\item [(2)] Condition $(\mathbf{C})$ holds for the Hermitian holomorphic vector bundle $\mathcal{E}$, with $\mathcal{E}=\bigvee f(w)$, via $\mathcal{H}_{k_i}$, $1\leq i\leq n$.
\item[(3)] There exist functions $\{\phi_i\}^{n-1}_{i=1} \subset GL({H}^{\infty}(\mathbb{B}_{m}))$ such that for all $1 \leq i < j \leq n$ and $w \in \mathbb{B}_{m}$,
$$\prod\limits_{k=i}^{j-1}\vert\phi_k(w)\vert^2\frac{\langle\widetilde{ T}_{i,j}t_j(w), t_i(w)\rangle}{\|t_{j}(w)\|^2}=\frac{\langle \widetilde{S}_{i,j} K_j(\cdot,\overline{w}), { K}_i(\cdot,\overline{w})\rangle}{\|{K}_{j}(\cdot,\overline{w})\|^2},$$
where $t_{n}(w) \in \ker(\mathbf{T}_{n}-w)$, $K_{n}(\cdot,\overline{w})\in\ker(\mathbf{S}_{n}-w),$
$t_{i}(w)=\widetilde{T}_{i, i+1}t_{i+1}(w)$, and ${K}_{i}(\cdot,\overline{w})=\widetilde{S}_{i,i+1}{K}_{i+1}(\cdot,\overline{w})$ for $1 \leq i \leq n-1.$
\item[(4)]$T_{k,i}\widetilde{T}_{i, j}=\widetilde{T}_{i, j}T_{k,j}$ and $S_{k,i}\widetilde{S}_{i, j}=\widetilde{S}_{i, j}S_{k,j}$ for all $1\leq i < j \leq n$ and $1\leq k\leq m$.
\end{enumerate}
Then $\mathbf{S}\sim_{s}\mathbf{T}.$
\end{theorem}
\begin{proof}
We can assume from condition (1) that $\mathcal{E}_{\mathbf{T}_{1}}= \mathcal{E}_{\mathbf{S}_{1}}\otimes \mathcal{E}$, where $\mathcal{E}(w)=\bigvee f(w)$ for analytic vector valued function $f$ over $\mathbb{B}_m$. Furthermore,
$\Vert t_1(w)\Vert=\Vert K_1(\cdot,\overline{w})\otimes f(w)\Vert$ for some $t_1(w)\in \ker(\mathbf{T}_1-w)$.
Note also that if we let $j=i+1$ in condition (3), then
\begin{equation}\label{equ2}
\vert\phi_i(w)\vert^2\frac{\|t_i(w)\|^2}{\|t_{i+1}(w)\|^2}
=\frac{\|K_i(\cdot,\overline{w})\|^2}{\|K_{i+1}(\cdot,\overline{w})\|^2}.
\end{equation}
 Thus, for $2 \leq i \leq n$,
 $$\begin{array}{llll}
\Vert t_i(w)\Vert^{2}&=&
\frac{\prod\limits_{k=1}^{i-1}\vert\phi_k(w)\vert^2\Vert K_i(\cdot,\overline{w})\Vert^{2}\Vert t_1(w)\Vert^{2}}{\Vert K_1(\cdot,\overline{w})\Vert^{2}}\\
&=&\frac{\prod\limits_{k=1}^{i-1}\vert\phi_k(w)\vert^2\Vert K_i(\cdot,\overline{w})\Vert^{2}\Vert  K_1(\cdot,\overline{w})\otimes f(w)\Vert^{2}}{\Vert K_1(\cdot,\overline{w})\Vert^{2}}\\
&=&\prod\limits_{k=1}^{i-1}\vert\phi_k(w)\vert^2\Vert K_i(\cdot,\overline{w})\Vert^{2}\Vert f(w)\Vert^{2}.
\end{array}$$

By using the Rigidity Theorem given in \cite{CD}, we now define the isometries $W_i$, $1 \leq i \leq n$, by
\begin{equation}\label{05}
W_1t_1(w):=K_1(\cdot,\overline{w})\otimes f(w) \,\,\text{ and }\,\,
W_it_i(w):=\prod\limits_{k=1}^{i-1}\phi_k(w) K_i(\cdot,\overline{w})\otimes f(w).
\end{equation}
Then, for $1 \leq i \leq m$,
\begin{equation}\label{equ3}
\begin{small}\begin{pmatrix}
T_{i,1}&\widetilde{T}_{1,2}&\cdots&\widetilde{T}_{1,n} \\
0&T_{i,2}&\cdots&\widetilde{T}_{2,n} \\
\vdots&\vdots&\ddots&\vdots\\
0&0&\cdots&T_{i,n}\\
\end{pmatrix}\end{small}\sim_{u}
\begin{small}\begin{pmatrix}
M^*_{z_{i}}\vert_{\mathcal{N}_1}&W_1\widetilde{T}_{1,2}W^*_2&\cdots&W_1\widetilde{T}_{1,n}W^*_n \\
0&M^*_{z_{i}}\vert_{\mathcal{N}_2}&\cdots&W_2\widetilde{T}_{2,n}W^*_n \\
\vdots&\vdots&\ddots&\vdots\\
0&0&\cdots&M^*_{z_{i}}\vert_{\mathcal{N}_n}\\
\end{pmatrix}\end{small},
\end{equation}
where $\mathcal{N}_j=\overline{\text{ran} W_j}$ for $1\leq j\leq n$.
Moreover, $$\ker(\mathbf{M}^*_z\vert_{\mathcal{N}_1}-w)=\bigvee \left\{K_1(\cdot, \overline{w})\otimes f(w)\right \}$$
and
$$\ker(\mathbf{M}^*_z\vert_{\mathcal{N}_j}-w)=\bigvee \left\{\prod\limits_{k=1}^{j-1}\phi_k(w)K_j(\cdot, \overline{w})\otimes f(w)\right\},$$ for $2\leq j\leq n$ and $w \in \mathbb{B}_{m}$.
Since $\mathbf{S}_{i} \sim_u (\mathbf{M}^*_{z}, \mathcal{H}_{k_i})$ for $1\leq i\leq n$,
$$\mathcal{K}_{\mathbf{S}_{i}}-\mathcal{K}_{\mathbf{T}_{i}}
=\mathcal{K}_{\mathbf{S}_{i}}-\mathcal{K}_{\mathbf{M}^*_z\vert_{\mathcal{N}_i}}=
\mathcal{K}_{\mathbf{S}_{i}}-(\mathcal{K}_{\mathbf{S}_{i}}+\mathcal{K}_{\mathcal{E}})=\sum \limits_{i,j=1}^{m}\frac{\partial^{2}\psi(w)}{\partial w_{i}\partial \overline{w}_{j}}dw_{i}\wedge d\overline{w}_{j},$$
where $\psi(w)=\log \Vert f(w)\Vert ^{2}$.
Next, by condition (2) and Lemma \ref{lem3.1},
there is a multiplier $(F^{\#})^{*}\in\text{Mult}(\mathcal{H}_{k_{i}}\otimes E,\mathcal{H}_{k_{i}}\otimes\mathbb{C})$ with $F\in H^{\infty}_{\mathbb{C}\rightarrow E}(\mathbb{B}_{m})$ that satisfies
$$M^*_{(F^{\#})^*}(K_{i}(\cdot,\overline{w})\otimes \lambda)=K_{i}(\cdot,\overline{w})\otimes F^{\#}(\overline{w})(\lambda)=K_{i}(\cdot,\overline{w})\otimes \lambda f(w),\quad\lambda\in\mathbb{C}, $$
 and a $G\in H^{\infty}_{E\rightarrow\mathbb{C}}(\mathbb{B}_{m})$ so that
$$\begin{array}{llll}
M^*_{(G^{\#})^*}M^*_{(F^{\#})^*}(K(\cdot,\overline{w})\otimes 1)&=&
M^*_{(G^{\#})^*}(K_{i}(\cdot,\overline{w})\otimes F^{\#}(\overline{w})(1))\\
&=&K_{i}(\cdot,\overline{w})\otimes G^{\#}(\overline{w})F^{\#}(\overline{w})(1)\\
&=&K_{i}(\cdot,\overline{w})\otimes 1,
\end{array}$$
where $E=\bigvee\limits_{w\in\mathbb{B}_{m}}f(w).$
Then there exist invertible operators $X_i\in \mathcal{L}(\mathcal{H}_{K_i},\mathcal{N}_i),1\leq i\leq n,$ such that
 $X_i\mathbf{S}_i=\mathbf{M}^*_z\vert_{\mathcal{N}_i}X_i=W_{i}\mathbf{T}_{i}W_{i}^{\ast}X_i.$
It then follows for some $g \in \mathcal{O}(\mathbb{B}_m)$, that
\begin{equation}\label{04}
X_1K_1(\cdot,\overline w)=g(w)K_1(\cdot,\overline w)\otimes f(w)\,\,\text{and}\,\, X_iK_i(\cdot,\overline w)=g(w)\prod\limits_{k=1}^{i-1}\phi_k(w)K_i(\cdot,\overline w)\otimes f(w),
\end{equation}
for all $2\leq i\leq n$ and $w \in \mathbb{B}_m$. A direct calculation shows that for $1\leq i\leq n-1$,
$$X_i\widetilde{S}_{i,i+1}=W_i\widetilde{T}_{i, i+1}W^*_{i+1}X_{i+1}.$$

Note that $T_{k,i}\widetilde{T}_{i, j}=\widetilde{T}_{i, j}T_{k,j}$ and
$S_{k,i}\widetilde{S}_{i, j}=\widetilde{S}_{i, j}S_{k,j}$ for all $1 \leq i < j \leq n$
and $1\leq k\leq m$. Moreover, there exist functions $\phi_{i,j}, \varphi_{i,j} \in \mathcal{O}(\mathbb{B}_m)$
such that $\widetilde{T}_{i,j}t_{j}(w)=\phi_{i,j}(w)t_i(w)$ and $\widetilde{S}_{i,j}K_{j}(\cdot,\overline{w})=\varphi_{i,j}(w)K_{i}(\cdot,\overline{w}).$
Therefore, for $ 1\leq i < j \leq n$, we get from condition (3) and (\ref{equ2}) that $$\prod_{k=i}^{j-1}\vert\phi_k(w)\vert^2\frac{\phi_{i,j}(w)\|t_i(w)\|^2}{\|t_{j}(w)\|^2}
=\frac{\varphi_{i,j}(w)\| K_i(\cdot,\overline{w})\|^2}{\|K_{j}(\cdot,\overline{w})\|^2}$$
and
$$\prod_{k=i}^{j-1}\vert\phi_k(w)\vert^2\frac{\|t_i(w)\|^2}{\|t_{j}(w)\|^2}=\frac{\| K_i(\cdot,\overline{w})\|^2}{\|K_{j}(\cdot,\overline{w})\|^2},$$ that is,  $\phi_{i,j}=\varphi_{i,j}$.
Moreover, by (\ref{05}) and (\ref{04}),
$$\begin{array}{llll}
W_{1}\widetilde{T}_{1,l}W^*_{l}X_{l}K_{l}(\cdot,\overline{w})&=&
W_{1}\widetilde{T}_{1,l}W^*_{l}\left(g(w)\prod\limits_{k=1}^{l-1} \phi_k(w)K_j(\cdot,\overline w)\otimes f(w)\right)\\
&=&g(w)W_{1}\widetilde{T}_{1,l}t_{l}(w)\\
&=&g(w)W_{1}(\phi_{1,l}(w)t_1(w))\\
&=&g(w)\phi_{1,l}(w)\left(K_l(\cdot, \overline{w})\otimes f(w)\right)\\
&=&X_1(\phi_{1,l}(w)K_{1}(\cdot,\overline{w}))\\
&=&X_{1}\widetilde{S}_{1,l}K_{l}(\cdot,\overline{w}),\end{array}$$
and
$$\begin{array}{llll}
W_{i}\widetilde{T}_{i,j}W^*_{j}X_{j}K_{j}(\cdot,\overline{w})&=&
W_{i}\widetilde{T}_{i,j}W^*_{j}\left(g(w)\prod\limits_{k=1}^{j-1} \phi_k(w)K_j(\cdot,\overline w)\otimes f(w)\right)\\
&=&g(w)W_{i}\widetilde{T}_{i,j}t_{j}(w)\\
&=&g(w)W_{i}(\phi_{i,j}(w)t_i(w))\\
&=&g(w)\phi_{i,j}(w)\left( \prod \limits_{k=1}^{i-1} \phi_k(w)K_i(\cdot,\overline{w})\otimes f(w)\right)\\
&=&X_i(\phi_{i,j}(w)K_{i}(\cdot,\overline{w}))\\
&=&X_{i}\widetilde{S}_{i,j}K_{j}(\cdot,\overline{w}),\end{array}
$$
for $1<l\leq n$ and $2\leq i<j\leq n$.
Hence, the operator $X:=\text{diag}(X_{1},\cdots,X_{n})$ is invertible and
$$\begin{array}{llll}
&&\begin{scriptsize}\begin{pmatrix}
M^*_{z_{i}}\vert_{\mathcal{N}_1}&W_1\widetilde{T}_{1,2}W^*_2&\cdots&W_1\widetilde{T}_{1,n}W^*_n \\
0&M^*_{z_{i}}\vert_{\mathcal{N}_2}&\cdots&W_2\widetilde{T}_{2,n}W^*_n \\
\vdots&\vdots&\ddots&\vdots\\
0& 0&\cdots&M^*_{z_i}\vert_{\mathcal{N}_n} \\
\end{pmatrix}\end{scriptsize}
\begin{scriptsize}\begin{pmatrix}
X_{1}&&&\\
&X_{2}&& \\
&&\ddots&\\
&&&X_{n}\\
\end{pmatrix}\end{scriptsize}\\
&=&\begin{scriptsize}\begin{pmatrix}
X_{1}&&&\\
&X_{2}&& \\
&&\ddots&\\
&&&X_{n}\\
\end{pmatrix}\end{scriptsize}
\begin{scriptsize}\begin{pmatrix}
S_{i,1}&\widetilde{S}_{1,2}&\cdots&\widetilde{S}_{1,n} \\
0&S_{i,2}&\cdots&\widetilde{S}_{2,n} \\
\vdots&\vdots&\ddots&\vdots\\
0&0&\cdots&S_{i,n}\\
\end{pmatrix}\end{scriptsize},
\end{array}$$
for all $1 \leq i \leq m$. From this and (\ref{equ3}), we conclude that $\mathbf{T}\sim_s \mathbf{S}$.
\end{proof}

\subsection{Uniqueness of strongly irreducible decomposition up to similarity}
When the Hilbert space $\mathcal{H}$ is finite-dimensional, the Jordan canonical form theorem indicates that every operator on $\mathcal{H}$ can be uniquely written as a direct sum of strongly irreducible operators up to similarity. Is there a corresponding analogue when one considers operators on an infinite-dimensional complex separable Hilbert space $\mathcal{H}$? The notion of a unicellular operator was introduced in \cite{MSB1956,MSB1968} and it was shown in \cite{GEK1967-2, GEK1967} that dissipative operators can be written as a direct sum of unicellular operators.
In \cite{BNF1979, BNF1975, NF1970, NF1970-3, NF1972, NF}, every $C_{0}$-operator on a complex separable Hilbert space was proven to be similar to a Jordan operator. Furthermore, in \cite{DH1990}, every bitriangular operator was shown to be quasisimilar to a Jordan operator. The concepts of strong irreducibility and of Banach irreducibility introduced in \cite{Gilfeather} and \cite{JS2006}, respectively, turned out to be equivalent. In \cite{Halmos}, the set of irreducible operators was proven to be dense in $\mathcal{L}(\mathcal{H})$ in the sense of Hilbert-Schmidt norm approximations. For the class $\mathcal{B}_{n}^{1}(\Omega)$, the work Y. Cao-J. S. Fang-C. L. Jiang \cite{CFJ2002}, C. L. Jiang \cite{J2004}, and C. L. Jiang-X. Z. Guo-K. Ji \cite{JGJ} involve the $K_{0}$-group of the commutant algebra as an invariant to show that an operator in $\mathcal{B}_{n}^{1}(\Omega)$ has a unique strong irreducible decomposition up to similarity.

Let $T\in\mathcal{B}_{1}^{1}(\mathbb{D})$ be an $n$-hypercontraction. Denote by $\mathcal{H}_{n}^1$ the Hilbert space of analytic functions on the unit disk $\mathbb{D}$ with reproducing kernel
$K(z,w)=\frac{1}{(1-z\overline{w})^{n}}, $ for $z, w \in \mathbb{D}$.
The results in \cite{DKT} and \cite{HJK} show that $\bigoplus\limits_{i=1}^{k}T$ is similar to the backward shift operator $\bigoplus\limits_{i=1}^{k}M_{z}^{*}$ on $\bigoplus\limits_{i=1}^{k}\mathcal{H}_{n}^1$ if and only if there exists a bounded subharmonic function $\varphi$ defined on $\mathbb{D}$ such that
$$\text{trace }{K}_{\bigoplus\limits_{i=1}^{k}M_{z}^{*}}(w)-\text{trace }{K}_{\bigoplus\limits_{i=1}^{k}T}(w)=k{K}_{M_{z}^{*}}(w)-k{K}_{T}(w) \leq \frac{\partial^{2}\varphi(w)}{\partial w\partial \overline{w}},\quad w\in\mathbb{D}.$$
Note here that if $\bigoplus\limits_{i=1}^{k}T\sim_{s}\bigoplus\limits_{i=1}^{k}M_{z}^{*}$, then $T\sim_{s}M_{z}^{*}$.

We start with some definitions given in \cite{Gilfeather, Halmos}.

\begin{definition}
Let $\{T\}^{\prime}=\{X\in\mathcal{L}(\mathcal{H}):XT=TX\}$ be the commutant of $T\in\mathcal{L}(\mathcal{H})$. The operator $T$
is called \emph{strongly irreducible} if $\{T\}^{\prime}$ does not have any nontrivial idempotents. It is called \emph{irreducible} if $\{T\}^{\prime}$ does not any contain nontrivial self-adjoint idempotents.
\end{definition}

\begin{definition}
Consider $T=(M^{*}_{z}, \mathcal{H}_{K}, K) \in \mathcal{B}_{n}^{1}(\Omega)$,
where $\mathcal{H}_{K}$ is an analytic function space with reproducing kernel $K$. Suppose that $T$ has a strongly irreducible decomposition, that is,
$$T=\bigoplus\limits_{i=1}^{t}T_{i}^{(r_{i})},$$
where each $T_{i}\in \mathcal{B}_{n_{i}}^{1}(\Omega)$ is strongly irreducible and $T_{i}\nsim_{s}T_{j}$ for $i\neq j$. It is said to have a \emph{unique strongly irreducible decomposition up to similarity}, if for any operator $\widetilde{T}$ that is similar to $T$ with a strongly irreducible decomposition
$$\widetilde{T}=\bigoplus\limits_{i=1}^{s}\widetilde{T}_{i}^{(l_{i})},$$
\begin{itemize}
  \item[(1)] $t=s$; and
  \item[(2)] there is a permutation $\pi$ on $\{1,2,\cdots,t\}$
such that $T_{i}\sim_{s}\widetilde{T}_{\pi(i)}$ and $r_{i}=l_{\pi(i)}, \ 1\leq i\leq t$.
\end{itemize}
\label{dingyi}
\end{definition}

\begin{proposition}\label{prop2}
Let $T=T_{1}\oplus\cdots\oplus T_{k},
S^{*}=S_{1}^{*}\oplus\cdots\oplus S_{k}^{*}\in\mathcal{B}_n^1(\Omega)$,  where $S_{i}^{*}$ is the adjoint of the operator of multiplication by $z$ on a reproducing kernel Hilbert space $\mathcal{H}_{K_{i}}$ with reproducing kernel $K_{i}$ and $S_{i}^{*}$ is strongly irreducible for $1\leq i\leq k$. Suppose that there exist an isometry $V$ and a Hermitian holomorphic vector bundle $\mathcal{E}$ over $\Omega$ for which condition $(\mathbf{C})$ holds such that
$$
VK_{{T},w^{i}\overline{w}^{j}}V^{*}=K_{{S^{*}},w^{i}\overline{w}^{j}}+K_{\mathcal{E},w^{i}\overline{w}^{j}}\otimes I_{n},\quad 0\leq i,j \leq n-1.
$$
Then, there is a permutation $\pi$ on $\{1,\cdots,k\}$ such that $S^{*}_{i}\sim_{s}T_{\pi(i)}$ for $1\leq i\leq k.$
\end{proposition}
\begin{proof}
Note first that $T, S^* \in \mathcal{B}_n^1(\Omega)$. By Lemma \ref{lem3.1.1}, we have
$$VK_{{T},w^{i}\overline{w}^{j}}V^{*}=K_{{S^{*}},w^{i}\overline{w}^{j}}+K_{\mathcal{E},w^{i}\overline{w}^{j}}\otimes I_{n}=K_{\mathcal{E}_{S^{*}}\otimes\mathcal{E},w^{i}\overline{w}^{j}},\quad 0\leq i,j \leq n-1.$$

From \cite{CD}, we also know that $\mathcal{E}_{T}$ is congruent to $\mathcal{E}_{S^{*}}\otimes\mathcal{E}$. Moreover, an operator $T\in \mathcal{B}_{n}^{1}(\Omega)$ has an unique irreducible decomposition up to unitary equivalence as is shown in \cite{JJ}. In \cite{CD}, it is proven that operators in $\mathcal{B}_{n}^{1}(\Omega)$ are unitary equivalent if and only if their holomorphic eigenvector bundles are equivalent. Since operators $S_{1}^{*},\cdots,S_{k}^{*}$ are strongly irreducible, there is a permutation $\pi$ on $\{1,\cdots,k\}$ so that $\mathcal{E}_{S_{i}^{*}}\otimes\mathcal{E}$ is congruent to $\mathcal{E}_{T_{\pi(i)}}$, that is, there are unitary operators $U_{i}, 1\leq i\leq k,$ such that
$$
U_{i}(\mathcal{E}_{S_{i}^{*}}\otimes\mathcal{E})(w)=\mathcal{E}_{T_{\pi(i)}}(w),\quad w\in\Omega.
$$

Now since the complex bundle $\mathcal{E}$ satisfies condition $(\mathbf{C})$, we obtain from Theorem \ref{thm3.2} an invertible operator
 $M^{*}_{(F^{\#})^{*}}$ with $F\in H^{\infty}_{\mathbb{C}\rightarrow E}(\Omega)$ such that
$$
M^{*}_{(F^{\#})^{*}}(\mathcal{E}_{S_{i}^{*}}(w))=(\mathcal{E}_{S_{i}^{*}}\otimes \mathcal{E})(w),\quad w\in\Omega.
$$
Hence, the operator $U_{i}M^{*}_{(F^{\#})^{*}}$ is invertible and
$$(U_{i}M^{*}_{(F^{\#})^{*}})(\mathcal{E}_{S_{i}^{*}}(w))=\mathcal{E}_{T_{\pi(i)}}(w),\quad w\in\Omega,$$
so that $S_{i}^{*}\sim_{s}T_{\pi(i)}$ for $1\leq i\leq k.$
\end{proof}

\begin{remark}
As in the proof of Theorem \ref{thm3.2}, Proposition \ref{prop2} shows that $T$ is similar to $S^{*}$. The result remains valid for operator tuples in $\mathcal{B}_n^m(\Omega)
$ possessing a unique irreducible decomposition up to unitary equivalence.
\end{remark}
\section{Applications of Theorem \ref{thm3.1}}
Theorem \ref{thm3.1} yields several sufficient conditions involving curvature for the similarity of certain adjoints of multiplication
tuples. For $m>1$, let $\Omega$ be a bounded domain in $\mathbb{C}^m$. The space $\mathcal{C}^{2}(\Omega)$ consists of functions defined on $\Omega$ whose second order partial derivatives are continuous. The reader is referred to \cite{LAA, PL, KO} for the following definitions and results.
\begin{definition}
A function $u\in \mathcal{C}^{2}(\Omega)$ is said to be \emph{pluriharmonic} if it satisfies the $m^{2}$ differential equations
$\frac{\partial ^{2}u}{\partial w_{i}\partial\overline{w}_{j}}=0$ for $1\leq i,j\leq m.$
\end{definition}
\begin{definition}
A real-valued function $u:\Omega\rightarrow \mathbb{R}\cup\{-\infty\} (u\not\equiv-\infty)$ is said to be \emph{plurisubharmonic} if
\begin{itemize}
  \item [(1)] $u(z)$ is upper-semicontinuous on $\Omega$; and
  \item [(2)] for each $z_{0}\in \Omega$ and some $z_{1}\in \mathbb{C}^{m}$ that is dependent on $z_{0}$, the function $u(z_{0}+\lambda z_{1})$ is subharmonic with respect to $\lambda\in\mathbb{C}$.
\end{itemize}
\end{definition}
\begin{definition}\label{def2.6}
A non-negative function $g:\mathbb{R}^{m}\rightarrow[0,+\infty)$ is called \emph{log-plurisubharmonic} if the function $\log g$ is plurisubharmonic.
\end{definition}
\begin{lemma}\label{lem2.8}
If $f$ is a pluriharmonic function on $\Omega$, then both $\log\vert f\vert$ and $\vert f\vert^{p} (0<p<\infty)$ are plurisubharmonic functions on $\Omega$.
\end{lemma}

\begin{lemma}\label{lem2.4}
A real-valued function $f \in \mathcal{C}^2(\Omega)$ is  plurisubharmonic if and only if $\left(\frac{\partial ^{2}f(w)}{\partial w_{i}\partial\overline{w}_{j}} \right)_{i,j=1}^{m}\geq0$ for every $w \in \Omega$.
\end{lemma}

\begin{lemma}\label{lem2.5}
Let $\varphi:\mathbb{R}^{2}\rightarrow\mathbb{R}$ be a convex function of two variables, increasing in each variable. If $F$ and $G$ are plurisubharmonic functions, then $\varphi(F,G)$ is also plurisubharmonic.
\end{lemma}

\begin{lemma}\label{lemma2.7}
Let $f$ and $g$ be log-plurisubharmonic functions. Then $f+g$ is also log-plurisubharmonic.
\end{lemma}
\begin{proof}
Since the mapping $k\mapsto\log(1+e^{k})$ is convex,
$$\varphi(x,y):=\log(e^{x}+e^{y})=x+\log(1+e^{y-x})$$
is also a convex function of two variables, increasing in each variable. Since $F:=\log f$ and $G:=\log g$ are plurisubharmonic, by Lemma \ref{lem2.5},
$$\varphi(F,G)=\log(e^{F}+e^{G})=\log(f+g)$$
is plurisubharmonic.
\end{proof}
Finally, we need the following definition of an \emph{n-hypercontraction} given in \cite{MV1993}:
\begin{definition}
Let $\mathbf{T}=(T_{1},\cdots,T_{m}) \in \mathcal{L}(\mathcal{H})^m$ be an $m$-tuple of commuting operators. Define an operator
$\mathbf{M}_{\mathbf{T}}:\mathcal{L}(\mathcal{H})\rightarrow\mathcal{L}(\mathcal{H})$
by
$$\mathbf{M}_{\mathbf{T}}(X)=\sum_{i=1}^m T_{i}^*XT_{i}, \quad X \in \mathcal{L}(\mathcal{H}).$$
An $m$-tuple $\mathbf{T}\in\mathcal{L}(\mathcal{H})^{m}$ is called an \emph{$n$-hypercontraction} if
$$\triangle_{\mathbf{T}}^{(l)}:=(I-\mathbf{M}_{\mathbf{T}})^{l}(I)\geq 0,$$
for all integers $l$ with $1 \leq l \leq n$.  The special case of a  $1$-hypercontraction corresponds to the usual row contraction.
\end{definition}
Note that since
$\mathbf{M}_{\mathbf{T}}^{l}(X)$
$=\sum\limits_{\alpha\in\mathbb{N}_{0}^{m}\atop\vert\alpha\vert=l}\frac{l!}{\alpha!}\mathbf{T}^{*\alpha}X\mathbf{T}^{\alpha}$ for all $l\geq0$,
one has
$$\triangle_{\mathbf{T}}^{(l)}=\sum\limits_{j=0}^l (-1)^j{l \choose j}\mathbf{M}_{\mathbf{T}}^{j}(I)=\sum \limits_{\vert\alpha\vert\leq l}(-1)^{\vert\alpha\vert}\frac{l!}{\alpha!(l-\vert\alpha\vert)!}\mathbf{T}^{*\alpha}\mathbf{T}^{\alpha}.$$
We now give a necessary and sufficient condition for the similarity of certain operator tuples in $\mathcal{B}_{1}^{m}(\mathbb{B}_{m})$ in terms of curvature and plurisubharmonic functions.

\begin{theorem}\label{thm5.2}
Let $\mathbf{T}=(T_{1},\cdots,T_{m}) \in \mathcal{B}_{1}^{m}(\mathbb{B}_{m})$ be an operator tuple on a Hilbert space $\mathcal{H}_{K}$ with reproducing kernel $K(z,w)=\sum\limits_{i=0}^{\infty}a(i)(z_{1}\overline{w}_{1}+\cdots+z_{m}\overline{w}_{m})^{i},$ where $a(i)>0.$ Consider $\mathbf{M}_{z}^{*}=(M_{z_{1}}^{*},\cdots,M_{z_{m}}^{*})$ on a weighted Bergman space $\mathcal{H}_{k}$ with $k > m+1$.
Suppose that $\mathbf{T}$ is $k$-hypercontractive and that $\lim\limits_{j}f_{j}(\mathbf{T}^{*}, \mathbf{T})h=0, h\in \mathcal{H}_K$, for
$f_j(z,w)=\sum\limits_{i=j}^{\infty}\mathbf{e}_i(z)(1-\langle z, w\rangle)^{k}\mathbf{e}_i(w)^*$
and an orthonormal basis $\{\mathbf{e}_i\}_{i=0}^{\infty}$ for $\mathcal{H}_{k}$.
Then $\mathbf{T}\sim_{s}\mathbf{M}_{z}^{*}$ if and only if there exists a bounded plurisubharmonic function $\psi$ such that
\begin{equation}\label{eqn}
  \mathcal{K}_{\mathbf{M}_{z}^{*}}(w)-\mathcal{K}_\mathbf{T}(w) = \sum \limits_{i,j=1}^{m}\frac{\partial^{2}\psi(w)}{\partial w_{i}\partial \overline{w}_{j}}dw_{i}\wedge d\overline{w}_{j},\quad w\in\mathbb{B}_{m}.
  \end{equation}
\end{theorem}
\begin{proof}
If $\mathbf{T}$ and $\mathbf{M}_{z}^{*}$ are similar, then there is a bounded invertible operator $X$ such that $X\mathbf{T}=\mathbf{M}_{z}^{*}X.$
If $\gamma$ is a non-vanishing holomorphic section of $\mathcal{E}_{\mathbf{T}}$, then $\widetilde{\gamma}=X{\gamma}$ is an non-vanishing holomorphic section of $\mathcal{E}_{\mathbf{M}_{z}^{*}}$,
and there are constants $a$ and $b$ such that
$$0<a\leq\frac{h_{\mathbf{T}}(w)}{h_{\mathbf{M}_{z}^{*}}(w)}=\Big\Arrowvert\frac{\gamma(w)}{\widetilde{\gamma}(w)}\Big\Arrowvert^{2}\leq b,\quad w\in\mathbb{B}_{m}.$$
Since $\frac{\gamma(w)}{\widetilde{\gamma}(w)}$ is pluriharmonic, Lemma \ref{lem2.8} shows that $\psi(w):=\log\Vert\frac{\gamma(w)}{\widetilde{\gamma}(w)}\Vert^{2}$ is a bounded plurisubharmonic function.

For the converse, let $\gamma$ and $\widetilde{\gamma}$ be non-vanishing holomorphic sections of $\mathcal{E}_{\mathbf{T}}$ and $\mathcal{E}_{\mathbf{M}_{z}^{*}}$, respectively.
Since $$\mathcal{K}_{\mathbf{M}_{z}^{*}}(w)-\mathcal{K}_\mathbf{T}(w)=\sum \limits_{i,j=1}^{m}\frac{\partial^{2}\log\Vert\frac{
\gamma(w)}{\widetilde{\gamma}(w)}\Vert^{2}}{\partial w_{i}\partial \overline{w}_{j}}dw_{i}\wedge d\overline{w}_{j},$$ condition (\ref{eqn}) gives
 $$\frac{\partial^{2}}{\partial w_{i}\partial \overline{w}_{j}}\left(\log\frac{\Vert\frac{\gamma(w)}{\widetilde{\gamma}(w)}\Vert^{2}}{e^{\psi(w)}}\right)=0,$$
 for all $1 \leq i, j \leq m$.
 Therefore, there exists a non-zero function $\phi \in \mathcal{O}(\mathbb{B}_m)$ such that $$\left\Vert\frac{\gamma(w)}{\phi(w)\widetilde{\gamma}(w)}\right\Vert^{2}=e^{\vert\psi(w)\vert}.$$ Since the function $\psi$ is bounded on $\mathbb{B}_m$, there exist constants $\widetilde{m}$ and $\widetilde{M}$ such that
\begin{equation}\label{051}
0<\widetilde{m}\leq\frac{\Vert\gamma(w)\Vert^{2}}{\Vert\phi(w)\widetilde{\gamma}(w)\Vert^{2}}\leq\widetilde{ M},\quad w\in\mathbb{B}_{m}.
\end{equation}

Note that since
$h_{\mathbf{T}}(w)=K(\overline{w},\overline{w})=\sum\limits_{\alpha\in\mathbb{N}_{0}^{m}}a(\vert\alpha\vert)\frac{\vert\alpha\vert!}{\alpha!}w^{\alpha}\overline{w}^{\alpha}
=\sum\limits_{\alpha\in\mathbb{N}_{0}^{m}}\rho(\alpha)w^{\alpha}\overline{w}^{\alpha}$,
$$\mathbf{T}^{*\alpha}\mathbf{T}^{\alpha}\mathbf{e}_{\beta}=\begin{cases}\frac{\rho(\beta-\alpha)}{\rho(\beta)}\mathbf{e}_{\beta},\quad \text{for }\alpha\leq\beta, \\ 0, \,\,\,\quad\quad\quad\quad\text{ otherwise,} \end{cases}$$
where $\rho(\alpha)=a(\vert\alpha\vert)\frac{\vert\alpha\vert!}{\alpha!}$ and $\{{\mathbf{e}}_{\alpha}\}_{\alpha\in\mathbb{N}_{0}^{m}}$ denotes an orthonormal basis for $\mathcal{H}_K$.
Since the reproducing kernel $\widetilde{K}(z,w)$ on a weighted Bergman space $\mathcal{H}_{k}$ with $k > m+1$ satisfies
$\frac{1}{\widetilde{K}}(\overline{w},\overline{w})=\sum\limits_{i=0}^{k}b(i)\vert w\vert^{2i}$, where $b(i)=\frac{(-1)^{i}k!}{i!(k-i)!}$, a direct calculation yields
$$\begin{array}{lll}
\frac{1}{\widetilde{K}}(\mathbf{T}^{*},\mathbf{T})\mathbf{e}_{\beta}&=&\sum\limits_{\mbox{\tiny$\begin{array}{c}
  \alpha\in\mathbb{N}_{0}^{m}\\
  \vert\alpha\vert\leq k\end{array}$}}
b(\vert\alpha\vert)\frac{\vert\alpha\vert!}{\alpha!}\mathbf{T}^{*\alpha}\mathbf{T}^{\alpha}\mathbf{e}_{\beta}\\
&=&\sum\limits_{\mbox{\tiny$\begin{array}{c}
  \alpha\in\mathbb{N}_{0}^{m}\\
  \vert\alpha\vert\leq k\\
  \alpha\leq\beta\end{array}$}}
b(\vert\alpha\vert)\frac{\vert\alpha\vert!}{\alpha!}\frac{\rho(\beta-\alpha)}{\rho(\beta)}\mathbf{e}_{\beta}\\
&=&\sum\limits_{\mbox{\tiny$\begin{array}{c}
  \alpha\in\mathbb{N}_{0}^{m}\\
  \vert\alpha\vert\leq k\\
  \alpha\leq\beta\end{array}$}}
b(\vert\alpha\vert)\frac{a(\vert\beta-\alpha\vert)}{a(\vert\beta\vert)}\frac{\vert\alpha\vert!\vert\beta-\alpha\vert!\beta!}{\alpha!\vert\beta\vert!(\beta-\alpha)!}\mathbf{e}_{\beta}.
\end{array}$$
Then,
$$
\frac{1}{\widetilde{K}}(\mathbf{T}^{*},\mathbf{T})\mathbf{e}_{\beta}=
\begin{cases}
\sum\limits_{i=0}^{s}b(i)\frac{a(s-i)}{a(s)}\mathbf{e}_{\beta},\quad \text{if}\,\,\beta=(s,0,\cdots,0),0\leq s\leq k,\\
\sum\limits_{i=0}^{k}b(i)\frac{a(s-i)}{a(s)}\mathbf{e}_{\beta},\quad \text{if}\,\,\beta=(s,0,\cdots,0),s> k.
\end{cases}
$$
Note that since $\frac{1}{\widetilde{K}}(\mathbf{T}^{*}, \mathbf{T})\geq0$ and $a(s)>0$ for every $s\geq0$,
$\sum\limits_{i=0}^{s}b(i)a(s-i)\geq0$ when $0 \leq s \leq k$
and $\sum\limits_{i=0}^{k}b(i)a(s-i)\geq0$ otherwise. Moreover,
$$\begin{array}{lll}
\frac{h_{\mathbf{T}}(w)}{h_{\mathbf{M}_{z}^{*}}(w)}
&=&
\left(\sum\limits_{j=0}^{k}b(j)\vert w\vert^{2j}\right)\left(\sum\limits_{i=0}^{\infty}a(i)\vert w\vert^{2i}\right)\\
&=&\sum\limits_{l=0}^{k}\sum\limits_{j=0}^{l}b(j)a(l-j)\vert w\vert^{2l}+\sum\limits_{l=k+1}^{\infty}\sum\limits_{j=0}^{k}
b(j)a(l-j)\vert w\vert^{2l}.
\end{array}$$

We next claim that there exists a constant $M'$ such that for all $w\in\mathbb{B}_{m},$
$$0<\frac{h_{\mathbf{T}}(w)}{h_{\mathbf{M}_{z}^{*}}(w)}=\left \Vert\frac{\gamma(w)}{\widetilde{\gamma}(w)} \right \Vert^2 \leq M'.$$ If not, then
$\left \Vert\frac{\gamma(w)}{\widetilde{\gamma}(w)} \right \Vert^2\rightarrow \infty$ as $\vert w\vert\rightarrow 1$ so that by (\ref{051}),  $\frac{1}{\vert\phi(w)\vert}\rightarrow0$ as $\vert w\vert\rightarrow 1.$ The maximum modulus principle would then imply that $\frac{1}{\vert\phi(w)\vert}=0$, a contradiction.

Finally, we have from Lemma \ref{2112.14} that $\frac{h_{\mathbf{T}}(w)}{h_{\mathbf{M}_{z}^{*}}(w)}=\|\mathcal{D}_{\mathbf{T}}\widetilde{\gamma}(w)\|^2,$ where $\mathcal{D}_{\mathbf{T}}=\frac{1}{\widetilde{K}}(\mathbf{T}^{*}, \mathbf{T})^{\frac{1}{2}}$ denotes the defect operator corresponding to $\mathbf{T}$ and
$$\sup\limits_{w\in\mathbb{B}_m}\frac{\Vert \mathcal{D}_{\mathbf{T}}\widetilde{\gamma}(w)\Vert^{2}}{\vert\langle \mathcal{D}_{\mathbf{T}}\widetilde{\gamma}(w),\zeta_{0}\rangle\vert^{2}}\leq\frac{M'}{a(0)b(0)}<\infty,$$
for some unit vector $\zeta_0 \in \overline{\text{ran } \mathcal{D}_{\mathbf{T}}}$. Using Theorem \ref{thm3.1}, we then conclude that $\mathbf{T}\sim_{s}\mathbf{M}_{z}^{*}$.
\end{proof}

The notion of curvature can also be used to describe the similarity of non-contractions as the following results show.
\begin{proposition}\label{theorem4.1}
Let $T\in \mathcal{B}_{1}^{1}(\mathbb{D}) \in \mathcal{L}(\mathcal{H})$. Suppose that $\{\phi_{j}\}_{j=0}^{m} \subset H^{\infty}(\mathbb{D})$ for $m>2$, $2\vert\phi_{j}(w)\vert^{2}>m(m+1)\vert\phi_{j}'(w)\vert^{2},$ and
$${K}_{M_{z}^*}(w)-{K}_T(w)=\frac{\partial^2}{\partial w\partial\overline{w}}\log \left[(1-\vert w\vert^{2})\sum_{j=0}^{m}\vert\phi_{j}(w)\vert^{2}+1 \right],\quad w\in\mathbb{D},$$
where $M_{z}$ is the operator of multiplication by $z$ on the Hardy space $H^2(\mathbb{D})$. Then $T\sim_{s}M_{z}^{*}$ but $T$ is not a contraction.
\end{proposition}

\begin{proof}
Set $\phi_{j}(w):=\sum\limits_{i=0}^{\infty}a_{ji}w^{i}$ and
$M:=\max\{\Vert\phi_{0}\Vert_{H^{2}},\Vert\phi_{1}\Vert_{H^{2}},\ldots,\Vert\phi_{m}\Vert_{H^{2}}\}.$ Denote by $\{e_{i}\}_{i=0}^{\infty}$ an orthonormal basis for the space $\mathcal{H}$ and define an operator $X$ as
$Xe_{n}:=\sum\limits_{j=0}^{m}a_{jn}e_{j}$ for $n\geq 0$. Let $y=\sum\limits_{n=0}^{\infty}b_ne_{n}\in\mathcal{H}.$ Then,
$$\begin{array}{lll}
\Vert X\Vert
&=&\sup\limits_{\Vert y\Vert=1}\Vert Xy\Vert
=\sup\limits_{\Vert y\Vert=1}\left\Vert X\sum\limits_{n=0}^{\infty}b_ne_{n}\right\Vert\\
&=&\sup\limits_{\Vert y\Vert=1}\left \Vert\sum\limits_{j=0}^{m}\sum\limits_{n=0}^{\infty}b_na_{jn}e_{j}\right \rVert\\
&\leq&\left(\sum\limits_{n=0}^{\infty}b^{2}_n\right)^{\frac{1}{2}}\sum\limits_{j=0}^{m}\left(\sum\limits_{n=0}^{\infty}a^{2}_{jn}\right)^{\frac{1}{2}}
\leq (m+1)M,
\end{array}$$
and therefore, $X$ is bounded.
For a non-vanishing holomorphic section $K(z,\overline{w})=\frac{1}{1-zw}$ of $\mathcal{E}_{M_{z}^*}$, we have
$$XK(z,\overline{w})=\sum_{i=0}^{\infty}w^{i}X(z^{i})=\sum_{j=0}^{m}\left(\sum_{i=0}^{\infty}a_{ji}w^{i}\right)z^{j}=\sum_{j=0}^{m}\phi_{j}(w)z^{j}.$$
This implies that
$\Vert XK(\cdot,\overline{w})\Vert^{2}=\sum\limits_{j=0}^{m}\vert\phi_{j}(w)\vert^{2}$
and hence,
\begin{equation}\label{90}
\begin{array}{lll}
{K}_{M_{z}^*}(w)-{K}_T(w)
&=&\frac{\partial ^{2}}{\partial w\partial\overline{w} }\log\left[(1-\vert w\vert^{2})\sum\limits_{i=0}^{m}\vert\phi_{i}(w)\vert^{2}+1\right]\\
&=&\frac{\partial ^{2}}{\partial w\partial\overline{w} }\log\left(\frac{\Vert XK(\cdot,\overline{w})\Vert^{2}}{\Vert K(\cdot,\overline{w})\Vert^{2}}+1\right).
\end{array}
\end{equation}

The operator $Y$ defined by $Y:=(I+X^{*}X)^{\frac{1}{2}}$ is invertible and ${K}_T={K}_{YM_z^*Y^{-1}}$. Therefore,  $T\sim_{u}YM_{z}^*Y^{-1}$, and hence, $T\sim_{s}M_{z}^*$.

Suppose now that $T$ is a contraction. For $w \in \mathbb{D}$, set $$\mathfrak{N}(w):=(1-\vert w\vert^{2})\sum_{i=0}^{m}\vert\phi_{i}(w)\vert^{2}.$$
Since for every $0 \leq j \leq m$, $$\left\vert\frac{\phi_{j}(w)}{\phi'_{j}(w)}\right\vert^{2}>\frac{m(m+1)}{2}>4,$$ we have $\vert\frac{\phi_{j}(w)}{\phi'_{j}(w)}+w\vert^{2}>1$, that is, $\vert\phi'_{j}(w)\vert^{2}-\vert(w\phi_{j}(w))'\vert^{2}<0.$
Then,
\begin{equation}\label{9}
\frac{\partial ^{2} \mathfrak{N}(w)}{\partial w\partial\overline{w} }=\frac{\partial ^{2}}{\partial w\partial\overline{w} }\left[(1-\vert w\vert^{2})\sum_{j=0}^{m}\vert\phi_{j}(w)\vert^{2}\right]
=\sum_{j=0}^{m}\left(\vert\phi'_{j}(w)\vert^{2}-\vert(w\phi_{j}(w))'\vert^{2}\right)
<0.
\end{equation}
Similarly, since  $\vert\frac{\phi_{j}'(w)}{\phi_{j}(w)}\vert^{2}<\frac{2}{m(m+1)}$ for $0 \leq j \leq m$, it follows that for all $w \in \mathbb{D},$
$$\begin{array}{lll}
\sum\limits_{0\leq i<j\leq m}\vert\frac{\phi_{j}'}{\phi_{j}}-\frac{\phi_{i}'}{\phi_{i}}\vert^{2}
&\leq&2\sum\limits_{0\leq i<j\leq m}\left(\vert\frac{\phi_{j}'}{\phi_{j}}\vert^{2}+\vert\frac{\phi_{i}'}{\phi_{i}}\vert^{2}\right)\\
&<&2\sum\limits_{0\leq i<j\leq m}\frac{4}{m(m+1)}=4\leq\frac{4}{(1-\vert w\vert^{2})^{2}}.
\end{array}$$
Therefore, for all $w \in \mathbb{D},$
$$\begin{array}{lll}
\sum\limits_{0\leq i<j\leq m}\frac{\vert\phi_{i}\phi_{j}'-\phi_{i}'\phi_{j}\vert^{2}}{(\sum\limits_{k=0}^{m}\vert\phi_{k}\vert^{2})^{2}}
&\leq&\sum\limits_{0\leq i<j\leq m}\frac{\vert\frac{\phi_{j}'}{\phi_{j}}-\frac{\phi_{i}'}{\phi_{i}}\vert^{2}}{(\vert\frac{\phi_{i}}{\phi_{j}}\vert+\vert\frac{\phi_{j}}{\phi_{i}}\vert)^{2}}\\
&\leq&\frac{1}{4}\sum\limits_{0\leq i<j\leq m}\vert\frac{\phi_{j}'}{\phi_{j}}-\frac{\phi_{i}'}{\phi_{i}}\vert^{2}<\frac{1}{(1-\vert w\vert^{2})^{2}}.
\end{array}$$
Then,
\begin{eqnarray}\nonumber
\frac{\partial ^{2} \log \mathfrak{N}(w)}{\partial w\partial\overline{w} }&=&-\frac{1}{(1-\vert w\vert^{2})^{2}}+\frac{(\sum\limits_{i=0}^{m}\vert\phi_{i}\vert^{2})(\sum\limits_{i=0}^{m}\vert\phi'_{i}\vert^{2})
-(\sum\limits_{i=0}^{m}\phi'_{i}\overline{\phi_{i}
})(\sum\limits_{i=0}^{m}\phi_{i}\overline{\phi'_{i}})}{(\sum\limits_{i=0}^{m}\vert\phi_{i}\vert^{2})^{2}} \\[4pt]\nonumber
&=&-\frac{1}{(1-\vert w\vert^{2})^{2}}+\sum\limits_{0\leq i<j\leq m}\frac{\vert\phi_{i}\phi_{j}'-\phi_{i}'\phi_{j}\vert^{2}}{(\sum\limits_{i=0}^{m}\vert\phi_{i}\vert^{2})^{2}}\\[4pt]\nonumber
&<&0. \\[4pt]\nonumber
\end{eqnarray}
However, since $\frac{\partial ^{2} \log \mathfrak{N}(w)}{\partial w\partial\overline{w} }=\frac{\mathfrak{N}(w)\frac{\partial ^{2} \mathfrak{N}(w)}{\partial w\partial\overline{w} }-\frac{\partial \mathfrak{N}(w)}{\partial w }\frac{\partial \mathfrak{N}(w) }{\partial\overline{w} }}{\mathfrak{N}^{2}(w)}<0,$
\begin{equation} \label{10}
\mathfrak{N}(w)\frac{\partial ^{2} \mathfrak{N}(w)}{\partial w\partial\overline{w} }-\frac{\partial \mathfrak{N}(w) }{\partial w }\frac{\partial \mathfrak{N}(w) }{\partial\overline{w} }<0.
\end{equation}
Finally, by (\ref{90})--(\ref{10}), it is easy to see that
$${K}_{M_{z}^*}(w)-{K}_T(w)=
\frac{\mathfrak{N}(w)\frac{\partial ^{2} \mathfrak{N}(w)}{\partial w\partial\overline{w} }-\frac{\partial \mathfrak{N}(w)}{\partial w }\frac{\partial \mathfrak{N}(w) }{\partial\overline{w} }+\frac{\partial ^{2} \mathfrak{N}(w)}{\partial w\partial\overline{w} }}{(\mathfrak{N}(w)+1)^{2}}
<0.$$
This contradicts the result given in \cite{M} that for a contraction $T \in \mathcal{B}_1^1(\mathbb{D})$, ${K}_{M_{z}^*}\geq {K}_T$.
\end{proof}
\begin{corollary}\label{cor4.3}
Let $T\in \mathcal{B}_{1}^{1}(\mathbb{D})$ and denote by $M_{z}$ the operator of multiplication by $z$ on the Hardy space $H^2(\mathbb{D})$. Suppose that  $\phi \in \mathcal{O}(\mathbb{D})$ is such that for all $w \in \mathbb{D}$, $\vert\phi(w)\vert^{2}>\vert\phi'(w)\vert$. If
$${K}_{M_{z}^*}(w)-{K}_T(w)=-\frac{(\vert\phi(w)\vert^{4}-\vert\phi'(w)\vert^{2})+\vert\phi(w)+w\phi'(w)\vert^{2}}{[\vert\phi(w)\vert^{2}(1-\vert w\vert^{2})+1]^{2}},\,\,\,\,w\in \mathbb{D},$$
then $T\sim_{s}M_{z}^{*}$, but $T$ is not a contraction.
\end{corollary}
\begin{corollary}\label{cor4.4}
Let $T\in \mathcal{B}_{1}^{1}(\mathbb{D})$ and denote by $M_{z}$ be the operator of  multiplication by $z$ on the Hardy space $H^2(\mathbb{D})$. Suppose that $\varphi \in \mathcal{O}(\mathbb{D})$ is such that for all $w \in \mathbb{D}$, $\vert\varphi(w)\vert>2\vert\varphi'(w)\vert.$ If
$${K}_{M_{z}^*}(w)-{K}_T(w)=\frac{\partial ^{2}}{\partial w\partial\overline{w} }\log \left[(1-\vert w\vert^{2})\vert\varphi(w)\vert^{2}+1\right],\,\,\,\,w\in \mathbb{D},
$$
then $T\sim_{s}M_{z}^{*}$, but $T$ is not a contraction.
\end{corollary}
In the following theorem, we will use log-plurisubharmonic functions to give a sufficient condition for the similarity of tuples in $\mathcal{B}_{1}^{m}(\Omega)$.
For an $m$-tuple $\mathbf{T}=(T_{1},\cdots,T_{m})\in\mathcal{L}(\mathcal{H})^m,$ let $\{\mathbf{T}\}':=\bigcap\limits_{j=1}^{m} \{T_j\}'.$
\begin{theorem}\label{theorem3.1}
Let $\mathbf{T}=(T_{1},\cdots,T_{m}),\mathbf{S}=(S_{1},\cdots,S_{m})\in \mathcal{B}_{1}^{m}(\Omega)$ be such that $\{\mathbf{S}\}'\cong H^{\infty}(\Omega)$. Suppose that
$$\mathcal{K}_{\mathbf{S}}(w)-\mathcal{K}_\mathbf{T}(w) = \sum\limits_{i,j=1}^{m}\frac{\partial ^{2}\psi(w)}{\partial w_{i}\partial\overline{w}_{j} }dw_{i}\wedge d\overline{w}_{j},\,\,\,\,w\in\Omega,$$
for some $\psi(w)=\log \sum\limits_{k=1}^{n}\vert\phi_{k}(w)\vert^{2}$, where $\phi_{k} \in \mathcal{O}(\Omega).$ If there exists an integer $l \leq n$ satisfying $\frac{\phi_k}{\phi_l} \in H^{\infty}(\Omega)$ for all $k \leq n$, then $\mathbf{T}\sim_{s}\mathbf{S}$, and ${K}_{\mathbf{T}}\leq {K}_\mathbf{S}$. In particular, when $m=1$, $\mathbf{T}$ and $\mathbf{S}$ are unitarily equivalent.
\end{theorem}
\begin{proof}
Since $\phi_{k} \in \mathcal{O}(\Omega)$ for all $1 \leq k \leq n$,
$\sum\limits_{i,j=1}^{m}\frac{\partial ^{2} \log\vert\phi_{k}(w)\vert^{2}}{\partial w_{i}\partial\overline{w}_{j} }=0.$
Therefore, by Lemmas \ref{lem2.4} and \ref{lemma2.7}, $\sum\limits_{k=1}^{n}\vert\phi_{k}(w)\vert^{2}$ is  log-plurisubharmonic, and
$${K}_{\mathbf{S}}(w)-{K}_\mathbf{T}(w)=\left(\frac{\partial ^{2}\psi(w)}{\partial w_{i}\partial\overline{w}_{j} }\right)_{i,j=1}^{m}\geq0.$$

Now let $t$ be a non-vanishing holomorphic section of $\mathcal{E}_\mathbf{S}$. Since $\{\mathbf{S}\}'\cong H^{\infty}(\Omega)$ and
$\{\frac{\phi_{k}}{\phi_{l}}\}_{k=1}^{n}\subset H^{\infty}(\Omega)$ for some $l\leq n$, we assume without loss of generality that $l = n$. Then there exist bounded operators $X_{i}\in\{\mathbf{S}\}', 1 \leq i \leq n-1$, such that
$$X_{i}t(w)=\frac{\phi_{i}(w)}{\phi_{n}(w)}t(w).$$
Define a linear operator
$X:\mathcal{H}\rightarrow \bigoplus\limits_{1}^{n-1} \mathcal{H}$ as $$Xh=\bigoplus\limits_{i=1}^{n-1}X_{i}h,$$ for $h\in \mathcal{H}$.
Since
$$\Vert X\Vert ^{2}=\sup_{\Vert h\Vert=1}\Vert Xh\Vert^{2}
=\sup_{\Vert h\Vert=1}\left(\sum\limits_{i=1}^{n-1}\Vert X_{i}h\Vert^{2}\right)
\leq\sum\limits_{i=1}^{n-1}\Vert X_{i}\Vert^{2} < \infty,$$
$X$ is a bounded linear operator. Furthermore, for any $w\in\Omega$,
\begin{eqnarray*}
&&\mathcal{K}_{\mathbf{S}}(w)-\mathcal{K}_\mathbf{T}(w)\\
&=&\sum\limits_{i,j=1}^{m}\frac{\partial ^{2}}{\partial w_{i}\partial\overline{w}_{j} }\log \left[\vert\phi_{n}(w)\vert^{2}\left(\sum\limits_{i=1}^{n-1}\vert\frac{\phi_{i}(w)}{\phi_{n}(w)}\vert^{2}+1\right)\right]dw_{i}\wedge d\overline{w}_{j}\\
&=&\sum\limits_{i,j=1}^{m}\frac{\partial ^{2}}{\partial w_{i}\partial\overline{w}_{j} }\log\left(\vert\phi_{n}(w)\vert^{2}\times\frac{\sum\limits_{i=1}^{n-1}\vert\frac{\phi_{i}(w)}{\phi_{n}(w)}\vert^{2}\Vert t(w)\Vert^{2}+\Vert t(w)\Vert^{2}}{\Vert t(w)\Vert^{2}}\right)dw_{i}\wedge d\overline{w}_{j}\\
&=&\sum\limits_{i,j=1}^{m}\frac{\partial ^{2}}{\partial w_{i}\partial\overline{w}_{j} }\log\left(\frac{\Vert Xt(w)\Vert^{2}}{\Vert t(w)\Vert^{2}}+1\right)dw_{i}\wedge d\overline{w}_{j}.
\end{eqnarray*}
Letting $Y:=(I+X^{*}X)^{\frac{1}{2}}$, we see that $Y$ is invertible and that $\mathcal{K}_{\mathbf{T}}=\mathcal{K}_{Y\mathbf{S}Y^{-1}}.$ Hence, $\mathbf{T} \sim_{u} Y\mathbf{S}Y^{-1}$ so that indeed, $\mathbf{T}\sim_{s}\mathbf{S}$.
\end{proof}
\section{On the Cowen-Douglas conjecture}

In \cite{CD}, M. J. Cowen and R. G. Douglas proved that for $T\in\mathcal{B}_{1}^{1}(\Omega)$, where $\Omega \subset \mathbb{C}$, the curvature $\mathcal{K}_{T}$ is a complete unitary invariant. Let $T,S\in\mathcal{B}_{1}^{1}(\mathbb{D})$ and suppose that the closure $\overline{\mathbb{D}}$ of $\mathbb{D}$ is a $K$-spectral set for $T$ and $S$. The Cowen-Douglas conjecture states that $T$ and $S$ are similar if and only if $$\lim\limits_{\vert w\vert\rightarrow 1}{K}_{T}^{i,i}(w)/{K}_{S}^{i,i}(w)=1.$$
Although the results of D. N. Clark and G. Misra in \cite{CM2,CM} show that the Cowen-Douglas conjecture is false, the connection between similarity and properties of holomorphic vector bundles merits further investigation especially since a one-sided implication of the conjecture holds is some specific cases. We describe a class of operators in $\mathcal{B}_{1}^{1}(\mathbb{D})$ for which the Cowen-Douglas conjecture is true.
\begin{example} \label{example}
 Let $T\in\mathcal{B}_{1}^{1}(\mathbb{D})$ and for $\lambda \geq 2$, consider $M_{z}^{*}$ on a weighted Bergman space $\mathcal{H}_{K}$ with reproducing kernel $K(z,w)=\frac{1}{(1-z\overline{w})^{\lambda}}$. Suppose that  $\mathcal{E}_T=\mathcal{E}_{M_{z}^{*}}\otimes\mathcal{E},$ where $\mathcal{E}(w)=\bigvee f(w)$, and
$\Vert f(w)\Vert^{2}$ is a polynomial in $\vert w\vert^{2}$. If condition $(\mathbf{C})$ holds for the Hermitian holomorphic vector bundle $\mathcal{E}$, then it follows from Lemma \ref{lem3.1.1} and Theorem \ref{thm3.2} that $T\sim_{s}M_{z}^{*}$.

Now, a direct calculation shows that
$$\frac{{K}_{T}(w)}{{K}_{M_{z}^{*}}(w)}=1+\frac{\frac{\partial ^{2}}{\partial w\partial\overline{w}}\log \Vert f(w)\Vert^{2}}{\frac{\partial ^{2}}{\partial w\partial\overline{w}}\log\frac{1}{(1-\vert w\vert^{2})^{\lambda}}}=1+\frac{(1-\vert w\vert^{2})^{2}}{\lambda}\frac{\partial ^{2}}{\partial w\partial\overline{w}}\log \Vert f(w)\Vert^{2}.$$
Since $\Vert f(w)\Vert^{2}$ is a polynomial in $\vert w\vert^{2}$, $\frac{\partial ^{2}}{\partial w\partial\overline{w}}\log \Vert f(w)\Vert^{2}$ is bounded above. Hence, $\lim\limits_{\vert w\vert\rightarrow 1}{K}_{T}(w)/{K}_{M_{z}^{*}}(w)=1.$
\end{example}
\begin{remark}
One can consider $\mathbf{T}=(T_{1},\cdots,T_{m})\in \mathcal{B}_{1}^{m}(\mathbb{B}_{m})$ and $\mathbf{M}_{z}^{*}=(M_{z_1}^{*},\cdots,M_{z_m}^{*})$ on a weighted Bergman space $\mathcal{H}_{k}$ with $k > m+1$ to obtain an operator tuple analogue of Example \ref{example}. Under the same assumptions on $\mathcal{E}$, $\mathbf{T} \sim_{s} {\mathbf{M}^*_z}$. Moreover, as $\vert w\vert \rightarrow 1,$
$$\frac{K^{i,i}_{\mathbf{T}}(w)}{K^{i,i}_{\mathbf{S}}(w)}=\frac{\frac{\partial^{2}}{\partial w_{i}\partial \overline{w}_{i}}\log\frac{\Vert f(w)\Vert^{2}}{(1-\vert w\vert^{2})^{k}}}{\frac{\partial^{2}}{\partial w_{i}\partial \overline{w}_{i}}\log\frac{1}{(1-\vert w\vert^{2})^{k}}}=
1+\frac{(1-\vert w\vert^{2})^{2}}{k(1-\vert w\vert^{2}+\vert w_{i}\vert^{2})}\frac{\partial ^{2}}{\partial w_{i}\partial\overline{w}_{i}}\log\Vert f(w)\Vert^{2}\rightarrow1.$$
\end{remark}
We next show that the Cowen-Douglas conjecture is false for tuples of commuting operators. As in \cite{CM}, we construct operator tuples in $\mathcal{B}^m_n(\mathbb{B}_m)$ for which the Cowen-Douglas conjecture holds; nevertheless, they are not similar. We begin with the following lemma given in \cite{HJX}:
\begin{lemma}\label{c4.2}
Let $\mathbf{T}=(T_{1},\cdots,T_{m}),\mathbf{S}=(S_{1},\cdots,S_{m})\in \mathcal{L}(\mathcal{H})^m$ be $m$-tuples of unilateral shift operators with nonzero weight sequences given by $\{\lambda_{\alpha}^{(1)},\cdots,\lambda_{\alpha}^{(m)}\}_{\alpha\in \mathbb{N}_{0}^{m}}$ and $\{\widetilde{\lambda}_{\alpha}^{(1)},\cdots,\widetilde{\lambda}_{\alpha}^{(m)}\}_{\alpha\in \mathbb{N}_{0}^{m}}$, respectively. Then $\mathbf{T}\sim_{s}\mathbf{S}$ if and only if there exist constants $C_1$ and $C_2$ such that
$$0<C_1\leq \left\vert\prod\limits_{k=0}^{l}\lambda_{\alpha+ke_{i}}^{(i)}/\prod\limits_{k=0}^{l}\widetilde{\lambda}_{\alpha+ke_{i}}^{(i)} \right\vert\leq C_2,$$
for every $l \in \mathbb{N}_0$, $\alpha\in \mathbb{N}_{0}^{m}$, and $1\leq i\leq m$.
\end{lemma}
\begin{example}
Consider the operator tuples $\mathbf{M}_{z}^{*}=(M_{z_1}^{*},\cdots,M_{z_m}^{*})$ on a reproducing kernel Hilbert space $\mathcal{H}_K$ with reproducing kernel $K(z,w)=\frac{1-\log(1-\langle z, w\rangle)}{1-\langle z, w\rangle}$ and on the Drury-Arveson space $\mathcal{H}_1$ with reproducing kernel $\widetilde{K}(z,w)=\frac{1}{1-\langle z, w\rangle}$. To simplify the notation, we will set $\mathbf{T}$ to be $\mathbf{M}_z^*$ on $\mathcal{H}_K$ while $\widetilde{\mathbf{T}}$
will denote the tuple on $\mathcal{H}_1$.
Since
$\widetilde{K}(\overline{w},\overline{w})=\frac{1}{1-\vert w\vert^{2}}=\sum\limits_{\alpha\in\mathbb{N}_{0}^{m}}\frac{\vert\alpha\vert!}{\alpha!}w^{\alpha}\overline{w}^{\alpha}$,
\begin{eqnarray}\nonumber
K(\overline{w},\overline{w})&=&\frac{1-\log(1-\vert w\vert^{2})}{1-\vert w\vert^{2}}\\[4pt]\nonumber
&=&\sum\limits_{i=0}^{\infty}\vert w\vert^{2i}\left(1-\sum\limits_{j=0}^{\infty}\frac{(-1)^{j}(-\vert w\vert^{2})^{j+1}}{j+1}\right)\\[4pt]\nonumber
&=&1+\sum\limits_{n=1}^{\infty}\left(1+\sum\limits_{i=1}^{n}\frac{1}{i}\right)\vert w\vert^{2n}\\[4pt]\nonumber
&=&\sum\limits_{\alpha\in\mathbb{N}_{0}^{m}}\left(1+\sum\limits_{i=1}^{\vert\alpha\vert}\frac{1}{i}\right)\frac{\vert\alpha\vert!}{\alpha!}
w^{\alpha}\overline{w}^{\alpha}.\nonumber
\end{eqnarray}

Now set $\widetilde{\rho}(\alpha)=\frac{\vert\alpha\vert!}{\alpha!}$ and $\rho(\alpha)=\left(1+\sum\limits_{i=1}^{\vert\alpha\vert}\frac{1}{i}\right)\frac{\vert\alpha\vert!}{\alpha!}$. Let $\{\mathbf{e}_{\alpha}\}$ and $\{\widetilde{\mathbf{e}_{\alpha}}\}$ denote orthonormal bases for $\mathcal{H}_K$ and $\mathcal{H}_1$, respectively.
From the relation between reproducing kernels and weight sequences, we have
$$T^{*}_{i}\mathbf{e}_{\alpha}=\sqrt{\frac{\rho(\alpha)}{\rho(\alpha+e_{i})}}\mathbf{e}_{\alpha+e_{i}},
\quad T_{i}\mathbf{e}_{\alpha}=\sqrt{\frac{\rho(\alpha-e_{i})}{\rho(\alpha)}}\mathbf{e}_{\alpha-e_{i}},$$
and
$$\widetilde{T}_i \widetilde{\mathbf{e}}_{\alpha}=\sqrt{\frac{\widetilde{\rho}(\alpha)}{\widetilde{\rho}(\alpha+e_{i})}}\widetilde{\mathbf{e}}_{\alpha+e_{i}},\quad
\quad \widetilde{T}_i \widetilde{\mathbf{e}}_{\alpha}=\sqrt{\frac{\widetilde{\rho}(\alpha-e_{i})}{\widetilde{\rho}(\alpha)}}\widetilde{\mathbf{e}}_{\alpha-e_{i}},$$
for $1\leq i\leq m$ and $e_{i}=(0,\cdots,1,\cdots,0)\in\mathbb{N}_{0}^{m}$ with $1$ in the $i$-th position.
Therefore, for every $\alpha\in\mathbb{N}_{0}^{m}$ and $1\leq i\leq m$,
$$\frac{\prod\limits_{k=0}^{l-1}\sqrt{\frac{\rho(\alpha+ke_{i})}{\rho(\alpha+(k+1)e_{i})}}}{\prod\limits_{k=0}^{l-1}\sqrt{\frac{\widetilde{\rho}(\alpha+ke_{i})}{\widetilde{\rho}(\alpha+(k+1)e_{i})}}}
=\sqrt{\frac{\rho(\alpha)\widetilde{\rho}(\alpha+le_{i})}{\rho(\alpha+le_{i})\widetilde{\rho}(\alpha)}}
=\sqrt{\frac{1+\sum\limits_{i=1}^{\vert\alpha\vert}\frac{1}{i}}{1+\sum\limits_{i=1}^{\vert\alpha\vert+l}\frac{1}{i}}}\rightarrow 0\quad \text{as}\,\,l\rightarrow\infty,$$
and by Lemma \ref{c4.2}, $\mathbf{T}$ and $\widetilde{\mathbf{T}}$ are not similar.

On the other hand, the definition of curvature yields
$$\mathcal{K}_{\mathbf{T}}(w)=-\sum \limits_{i,j=1}^{m}\frac{\partial^{2}}{\partial w_{i}\partial \overline{w}_{j}}\log\frac{1-\log(1-\vert w\vert^{2})}{1-\vert w\vert^{2}}dw_{i}\wedge d\overline{w}_{j},$$
and
$$\mathcal{K}_{\widetilde{\mathbf{T}}}(w)=-\sum \limits_{i,j=1}^{m}\frac{\partial^{2}}{\partial w_{i}\partial \overline{w}_{j}}\log \frac{1}{1-\vert w\vert^{2}}dw_{i}\wedge d\overline{w}_{j}.$$
Then as $\vert w\vert \rightarrow 1$,
\begin{eqnarray}\nonumber
\frac{K^{i,i}_{\mathbf{T}}(w)}{K^{i,i}_{\mathbf{M}_{z}^{*}}(w)}&=&\frac{\frac{\partial^{2}}{\partial w_{i}\partial \overline{w}_{i}}\log\frac{1-\log(1-\vert w\vert^{2})}{1-\vert w\vert^{2}}}{\frac{\partial^{2}}{\partial w_{i}\partial \overline{w}_{i}}\log \frac{1}{1-\vert w\vert^{2}}}\\[4pt]\nonumber
&=&1+\frac{\frac{\partial^{2}}{\partial w_{i}\partial \overline{w}_{i}}\log\left(1-\log(1-\vert w\vert^{2})\right)}{\frac{\partial^{2}}{\partial w_{i}\partial \overline{w}_{i}}\log \frac{1}{1-\vert w\vert^{2}}}\\[4pt]\nonumber
&=&1+\frac{1}{1-\log(1-\vert w\vert^{2})}-\frac{\frac{\vert w_{i}\vert^{2}}{1-\vert w\vert^{2}+\vert w_{i}\vert^{2}}}{\left(1-\log(1-\vert w\vert^{2})\right)^{2}}\\[4pt]\nonumber
&\rightarrow&1.\nonumber
\end{eqnarray}
\end{example}

\section{Further generalizations of single Cowen-Douglas operator results}

In this section, we extend additional results given for a single Cowen-Douglas operator to a tuple of commuting operators.

\subsection{Inequalities involving the trace of curvature}
In \cite{M}, G. Misra gave the curvature inequality for contractions in $\mathcal{B}_{1}^{1}(\mathbb{D})$. Later in \cite{BKM}, S. Biswas, D. K. Keshari, and G. Misra generalized the result and presented the following curvature matrix inequality for $m$-tuples in $\mathcal{B}_{1}^{m}(\mathbb{B}_{m})$:

\begin{lemma}\label{lemma41}
Let $\mathbf{T}\in\mathcal{B}_{1}^{m}(\mathbb{B}_{m})$ be a row-contraction and consider ${\mathbf{M}^*_z}$ on the Dury-Arveson space $\mathcal{H}_1$. Then $K_\mathbf{T}(w)\leq K_{\mathbf{M}_z^*}(w)$ for $w\in\mathbb{B}_{m}$.
\end{lemma}
The above inequality implies that one can use the curvature matrix to determine whether a tuple of operators in $\mathcal{B}_{1}^{m}(\mathbb{B}_{m})$ is a row-contraction. We derive  an analogous curvature matrix inequality for an $n$-hypercontraction in $\mathcal{B}^m_n(\mathbb{B}_{m})$. We first show that the trace of the curvature matrix for $\mathbf{T}=(T_{1},\cdots,T_{m})\in \mathcal{B}_{n}^{m}(\Omega)$ is independent of the choice of the holomorphic frame of $\mathcal{E}_{\mathbf{T}}$.
\begin{proposition}
Let $\mathbf{T}=(T_{1},\cdots,T_{m})\in \mathcal{B}_{n}^{m}(\Omega)$. Suppose that  $\sigma=\{\sigma_{1},\ldots,\sigma_{n}\}$ and $\widetilde{\sigma}=\{\widetilde{\sigma}_{1},\ldots,\widetilde{\sigma}_{n}\}$ are holomorphic frames of $\mathcal{E}_{\mathbf{T}}$. Then
$$K_{\mathbf{T}}^{i,j}(\widetilde{\sigma})=\phi^{-1}K_{\mathbf{T}}^{i,j}(\sigma)\phi\quad\text{and}\quad \text{trace }{K}_{\mathbf{T}}(\sigma)=\text{trace }{K}_{\mathbf{T}}(\widetilde{\sigma}),$$ for some invertible holomorphic matrix-valued function $\phi$ on $\Omega$.
\end{proposition}
\begin{proof}
Since $\sigma$ and $\widetilde{\sigma}$ are holomorphic frames of $\mathcal{E}_{\mathbf{T}}$,
there is an invertible holomorphic matrix $\phi=(\phi_{ij})_{i,j=1}^{n}$ such that for all $w \in \Omega$,
$(\widetilde{\sigma}_{1}(w),\cdots,\widetilde{\sigma}_{n}(w))=(\sigma_{1}(w),\cdots,\sigma_{n}(w))\phi(w).$
Therefore,
$$\begin{array}{lll}
\widetilde{h}(w)&=&(\langle\widetilde{\sigma}_{j}(w),\widetilde{\sigma}_{i}(w)\rangle)_{i,j=1}^{n}\\
&=&\left(\left\langle\sum\limits_{k=1}^{n}\phi_{kj}(w)\sigma_{k}(w),\sum\limits_{k=1}^{n}\phi_{ki}(w)\sigma_{k}(w)\right\rangle\right)_{i,j=1}^{n}\\
&=&{\phi^{*}}(w)h(w)\phi(w).
\end{array}$$
Since $\phi$ is holomorphic and invertible,
$$\begin{array}{lll}
K_{\mathbf{T}}^{i,j}(\widetilde{\sigma})(w)&=&-\frac{\partial}{\partial \overline{w}_{j}}\left[\left({\phi^{*}}(w)h(w)\phi(w)\right)^{-1}\frac{\partial}{\partial w_{i}}\left({\phi^{*}}(w)h(w)\phi(w)\right)\right]\\[4pt]
&=&-\left[\phi^{-1}(w)\frac{\partial h^{-1}(w)}{\partial \overline{w}_{j}}\frac{\partial h(w)}{\partial w_{i}}\phi(w)+\phi^{-1}(w)h^{-1}(w)\frac{\partial ^{2} h(w)}{\partial w_{i}\partial\overline{w}_{j} }\phi(w)\right]\\[4pt]
&=&\phi^{-1}(w)K_{\mathbf{T}}^{i,j}(\sigma)(w)\phi(w).
\end{array}$$
This shows that $K_{\mathbf{T}}^{i,j}(\widetilde{\sigma})(w)$ is similar to $K_{\mathbf{T}}^{i,j}(\sigma)(w)$ and that
$$\text{trace}\;K_{\mathbf{T}}^{i,j}(\widetilde{\sigma})(w)=\text{trace}\;K_{\mathbf{T}}^{i,j}(\sigma)(w).$$
Thus, for all $w \in \Omega$,
\begin{align}\nonumber
\text{trace } {K}_{\mathbf{T}}(\widetilde{\sigma})(w)=&-\sum \limits_{i,j=1}^{m}\text{trace }\left(\frac{\partial}{\partial \overline{w}_{j}}\left(\widetilde{h}^{-1}(w)\frac{\partial \widetilde{h}(w)}{\partial w_{i}} \right)\right)\\[4pt]\nonumber
=&-\sum \limits_{i,j=1}^{m}\text{trace }\left(\frac{\partial}{\partial \overline{w}_{j}}\left(h^{-1}(w)\frac{\partial h(w)}{\partial w_{i}} \right)\right)\\[4pt]\nonumber
=&\text{trace } {K}_{\mathbf{T}}(\sigma)(w).\nonumber
\end{align}
\end{proof}
In \cite{MV1993}, V. M\"{u}ller and F.-H. Vasilescu gave the following description of when an $n$-hypercontraction is unitarily equivalent to an adjoint of a multiplication tuple. Let $\mathbf{M}^{*}_{z,n,E}$ denote the operator tuple $\mathbf{M}_z^*$ on the space $\mathcal{H}_{n}\otimes E$:

\begin{lemma}\label{lem3.5}
Let $\mathbf{T}=(T_{1},\cdots,T_{m}) \in \mathcal{L}(\mathcal{H})^m$ be an $m$-tuple of commuting operators and let $n$ be a positive integer. Then there is a Hilbert space $E$ and an $\mathbf{M}^*_{z,n,E}$-invariant subspace $\mathcal{N}$ of $\mathcal{H}_{n}\otimes E$ such that $\mathbf{T}$ is unitarily equivalent to $\mathbf{M}^{*}_{z,n,E}\vert_{\mathcal{N}}$ if and only if $\mathbf{T}$ is an $n$-hypercontraction with $\lim\limits_{s\rightarrow\infty}\mathbf{M}_{\mathbf{T}}^{s}(I)=0$ in the strong operator topology.
\end{lemma}

Another important tool for our work in the current section comes from the following result by D. K. Keshari in \cite{D}:
\begin{lemma}\label{lem}
Let $\mathcal{E}$ be a Hermitian holomorphic vector bundle of rank $n$ over $\Omega$. Then the curvature matrices of the determinant bundle $\det \mathcal{E}$ and of the vector bundle $\mathcal{E}$ satisfy the equality
$${K}_{\det\mathcal{E}}=\text{trace }{K}_{\mathcal{E}}.$$
\end{lemma}

We now give a corresponding curvature inequality that holds for a $t$-hypercontractive tuple in the class $\mathcal{B}_{n}^{m}(\mathbb{B}_{m})$.
\begin{theorem}\label{theorem3.6}
Let $t \in \mathbb{N}$. For a $t$-hypercontraction $\mathbf{T}=(T_{1},\cdots,T_{m})\in \mathcal{B}_{n}^{m}(\mathbb{B}_{m})$,
$$K_{\det\mathcal{E}_{\mathbf{T}}}\leq K_{\det\mathcal{E}_{\bigoplus\limits_{1}^{n}\mathbf{M}_{z, t}^{*}}},$$
where ${\mathbf{M}_{z, t}^*}$ is on the space $\mathcal{H}_t$ with reproducing kernel $K(z,w)=\frac{1}{(1-\langle z,w\rangle)^{t}}$.
\end{theorem}
\begin{proof}
Since $\mathbf{T}\in \mathcal{B}_{n}^{m}(\mathbb{B}_{m})$ is a $t$-hypercontraction, by Lemma \ref{lem3.5}, there exist a Hilbert space $E$ and an $\mathbf{M}^*_{z,t,E}$-invariant subspace $\mathcal{N}\subset \mathcal{H}_{t}\otimes E$ such that $\mathbf{T}\sim_{u}\mathbf{M}^{*}_{z,t,E}\vert_{\mathcal{N}}$. Let
$$K(\cdot,\overline{w})\otimes e(\alpha, w)\in\ker (\mathbf{M}^{*}_{z, t,E}\vert_{\mathcal{N}}-w)=\bigcap \limits_{i=1}^{n} \ker (M_{z_{i}}^{*}\vert_{\mathcal{N}}-w_{i}).$$ If we denote by $\{a(\alpha)z^{\alpha}\}_{\alpha \in\mathbb{N}_{0}^{m}}$
an orthonormal basis for $\mathcal{H}_t \otimes E,$
then for every $f\in E(w)$, $\alpha'\in \mathbb{N}_{0}^{m}$, and $1\leq i\leq m$,
\begin{align}\nonumber
0=&\langle(M_{z_{i}}^{*}-w_{i})\sum_{\alpha\in\mathbb{N}_{0}^{m}}a(\alpha)\otimes e(\alpha, w)z^{\alpha}w^{\alpha}, a(\alpha')\otimes fz^{\alpha'}\rangle\\[4pt]\nonumber
=&\sum_{\alpha\in\mathbb{N}_{0}^{m}}\langle a(\alpha)\otimes e(\alpha, w)z^{\alpha}w^{\alpha},\,\, a(\alpha'+e_{i})\otimes fz^{\alpha'+e_{i}}\rangle\\[4pt]\nonumber
 &\qquad\qquad\qquad\qquad\qquad-\sum_{\alpha\in\mathbb{N}_{0}^{m}}\langle a(\alpha)\otimes e(\alpha, w)z^{\alpha}w^{\alpha+e_{i}}, a(\alpha')\otimes fz^{\alpha'}\rangle\\[4pt]\nonumber
=&\langle a(\alpha'+e_{i})\otimes e(\alpha'+e_{i}, w)z^{\alpha'+e_{i}}w^{\alpha'+e_{i}}, a(\alpha'+e_{i})\otimes fz^{\alpha'+e_{i}}\rangle\\[4pt]\nonumber
 &\qquad\qquad\qquad\qquad\qquad-\langle a(\alpha')\otimes e(\alpha',w)z^{\alpha'}w^{\alpha'+e_{i}}, a(\alpha')\otimes fz^{\alpha'}\rangle\\[4pt]\nonumber
=&\langle e(\alpha'+e_{i},w)-e(\alpha',w), f\rangle w^{\alpha'+e_{i}}.\nonumber
\end{align}
Hence, for every $\alpha,\beta\in \mathbb{N}_{0}^{m}$, $e(\alpha, w)=e(\beta, w)$, so that we can set $e(w):=e(\alpha, w)$. Then for all $ w\in\mathbb{B}_{m},$
$$\ker (\mathbf{M}^{*}_{z,t,E}\vert_{\mathcal{N}}-w)=\bigvee \{K(\cdot,\overline{w})\otimes e(w), e(w) \in E(w)\}.$$
Since $\dim \ker (\mathbf{M}^{*}_{z,t,E}\vert_{\mathcal{N}}-w)=\dim \ker (\mathbf{T}-w)=n,$ we can assume that $\{K(\cdot,\overline{w})\otimes e_{i}(w)\}_{i=1}^{n}$ is a basis for $\ker (\mathbf{M}^{*}_{z,t,E}\vert_{\mathcal{N}}-w)$ and $\mathcal{E}(w)=\bigvee\limits_{1\leq i\leq n}e_{i}(w)$. Therefore,
$$h(w)=K(\overline{w},\overline{w})\Bigg(\langle e_j(w),e_i(w)\rangle\Bigg)_{i,j=1}^{n}=K(\overline{w},\overline{w})h_{\mathcal{E}}(w),$$
and
$$
K_{\det\mathcal{E}_{\mathbf{T}}}
=\left(-\frac{\partial^2 \log(K(\overline{w},\overline{w}))^{n}}{\partial\overline{w}_{j}\partial w_{i}}-\frac{\partial^2 \log \left(\det h_{\mathcal{E}}(w)\right)}{\partial\overline{w}_{j}\partial w_i}\right)_{{i,j=1}}^{m}
=K_{\det\mathcal{E}_{\bigoplus \limits_{1} ^{n} \mathbf{M}_{z, t}^{*}}}+K_{\det\mathcal{E}}.
$$
Finally, since it is known from \cite{BKM} that $K_{\det\mathcal{E}}$ is negative definite, the proof is complete.
\end{proof}
We next show that the result of \cite{HJK} holds for the multi-operator case as well. For $\mathbf{T}=(T_{1},\cdots,T_{m})\in\mathcal{B}_{n}^{m}(\Omega) \subset \mathcal{L}(\mathcal{H})^m$, one defines a projection-valued function $\Pi : \Omega\rightarrow \mathcal{L}(\mathcal{H})$ that assigns to each $w\in\Omega,$ an orthogonal projection $\Pi(w)$ onto $\ker(\mathbf{T}-w)$.
Define an operator $\Gamma^{*}(w): \ker(\mathbf{T}-w)\rightarrow \mathbb{C}^{n}$ by
$$\Gamma^{*}(w)(f)=f(w),$$ where $f\in\ker(\mathbf{T}-w).$
Then, $h(w)=\Gamma^{*}(w)\Gamma(w)$ and $\Pi(w)=\Gamma(w)h^{-1}(w)\Gamma^{*}(w).$

\begin{theorem}\label{lem21}
Let $\mathbf{T}\in \mathcal{B}_{n}^{m}(\Omega)$ and let $\Pi(w)$ be an orthogonal projection onto $\ker(\mathbf{T}-w)$. Then for $w \in \Omega,$
$$\sum\limits_{i=1}^{m}\bigg{\Vert}\frac{\partial\Pi(w)}{\partial w_{i}}\bigg{\Vert}_{\mathfrak{S}_{2}}^{2}=-\text{trace } K_{\mathbf{T}}(w),$$
where $\mathfrak{S}_{2}$ denotes the Hilbert-Schmidt class of operators.
\end{theorem}
\begin{proof}
From $\Pi(w)=\Gamma(w)h^{-1}(w)\Gamma^{*}(w),$ we have for $w \in \Omega$,
\begin{align}\nonumber
\frac{\partial\Pi(w)}{\partial \overline{w}_{j}}\frac{\partial\Pi(w)}{\partial w_{i}}=&\Gamma(w)\Bigg[\frac{\partial h^{-1}(w)}{\partial \overline{w}_{j}}\frac{\partial h(w)}{\partial w_{i}} h^{-1}(w)+h^{-1}(w)\frac{\partial^{2} h(w)}{\partial \overline{w}_{j}\partial w_{i}}h^{-1}(w)\\[4pt]\nonumber
&\quad +\frac{\partial h^{-1}(w)}{\partial \overline{w}_{j}}h(w)\frac{\partial h^{-1}(w)}{\partial w_{i}}+h^{-1}(w)\frac{\partial h(w)}{\partial \overline{w}_{j}}\frac{\partial h^{-1}(w)}{\partial w_{i}} \Bigg] \Gamma^{*}(w)\\[4pt]\nonumber
=&\Gamma(w)\Bigg[\frac{\partial }{\partial \overline{w}_{j}}\left(h^{-1}(w)\frac{\partial h(w)}{\partial w_{i}}\right)h^{-1}(w)\\[4pt]\nonumber
 &\quad +\frac{\partial }{\partial \overline{w}_{j}}\left(h^{-1}(w)h(w)\right)\frac{\partial h^{-1}(w)}{\partial w_{j}}\Bigg]\Gamma^{*}(w)\\[4pt]\nonumber
=&\Gamma(w)\frac{\partial }{\partial \overline{w}_{j}}\left[h^{-1}(w)\frac{\partial h(w)}{\partial w_{i}}\right]h^{-1}(w)\Gamma^{*}(w)\\[4pt]\nonumber
=&-\Gamma(w)[K_{\mathbf{T}}^{i,j}(w)h^{-1}(w)]\Gamma^{*}(w)\nonumber.
\end{align}
Hence,
$$\begin{array}{lll}
\left\Vert\frac{\partial\Pi(w)}{\partial w_{i}}\right\Vert_{\mathfrak{S}_{2}}^{2}
&=&\text{trace } \left(\frac{\partial\Pi(w)}{\partial \overline{w}_{i}}\frac{\partial\Pi(w)}{\partial w_{i}}\right)\\
&=&-\text{trace } \left([K_{\mathbf{T}}^{i,i}(w)h^{-1}(w)]\Gamma^{*}(w)\Gamma(w)\right)\\
&=&-\text{trace }K_{\mathbf{T}}^{i,i}(w).
\end{array}$$
\end{proof}
\subsection{Final Remark}

For $T_1, T_2 \in \mathcal{L}(\mathcal{H})$, the \emph{Rosenblum operator} $\sigma_{T_{1},T_{2}}:\mathcal{L}(\mathcal{H})\rightarrow\mathcal{L}(\mathcal{H})$ is defined as
$$\sigma_{T_{1},T_{2}}(X)=T_{1}X-XT_{2},\quad X \in \mathcal{L}(\mathcal{H}).$$
 In \cite{Gilfeather}, F. Gilfeather showed that
an operator $T\in\mathcal{L}(\mathcal{H})$ is strongly irreducible if there is no non-trivial idempotent in $\{T\}'$. The following relationship between strong irreducibility and the class $\mathcal{FB}_2^1(\Omega)$ is given in \cite{JJDG}:

\begin{lemma}\label{lem2.10}
An operator $T=\begin{pmatrix} T_{1}\,\, & T_{1,2}\\ 0 \,\,& T_{2}\end{pmatrix} \in \mathcal{F}\mathcal{B}_{2}^{1}(\Omega)$ is strongly irreducible if and only if $T_{1,2}\notin \text{ ran } \sigma_{T_{1},T_{2}}.$
\end{lemma}
We conclude with an example illustrating the complexity of the similarity problem. It implies that the trace of the curvature matrix is not always a proper similarity invariant for the class $\mathcal{B}^m_n(\Omega)$.
Let $T_{1}$ and $T_{2}$ be backward shift operators on Hilbert spaces $\mathcal{H}_{K_1}$ and $\mathcal{H}_{K_2}$ defined on the unit disk $\mathbb{D}$ with reproducing kernels $K_{1}(z,w)=\frac{1}{1-\overline{w}z}$ and $K_{2}(z,w)=\frac{1}{(1-\overline{w}z)^{3}}$, respectively.
\begin{example}\label{example4.7} Consider
$\widetilde{T}=\begin{pmatrix} T_{1}\,\, & 0\\ 0\,\, & T_{2}\end{pmatrix}\in \mathcal{B}_{2}^{1}(\mathbb{D})$ and
$T=\begin{pmatrix} T_{1}\, & T_{1,2}\\ 0\, & T_{2} \end{pmatrix}\in \mathcal{F}\mathcal{B}_{2}^{1}(\mathbb{D})$. Suppose that $$T_{1,2}\notin \text{ ran }\sigma_{T_{1},T_{2}}(X)=\{T_{1}X-XT_{2}:~X\in\mathcal{L}(\mathcal{H}_{K_{2}},\mathcal{H}_{K_{1}})\}.$$

Let us first note that $t_{1}(w):=K_{1}(\cdot,\overline{w})\in \ker(T_1-w)$ and $t_{2}(w):=K_{2}(\cdot,\overline{w})\in \ker(T_2-w)$. It can be easily checked that $\{t_1,t_1'+t_2\}$ is a holomorphic frame of $\mathcal{E}_{T}$ and that
$$\begin{array}{lllll}
h_{T}(w)&=&\left(\begin{array}{ccccc}\frac{1}{1-\vert w\vert^{2}}&
\frac{\partial}{\partial w}\frac{1}{1-\vert w\vert^{2}}\\
\frac{\partial}{\partial\overline{w}}\frac{1}{1-\vert w\vert^{2}}&\quad
\frac{\partial^2}{\partial\overline{w}\partial w}\frac{1}{1-\vert w\vert^{2}}
+\frac{1}{(1-\vert w\vert^{2})^{3}}
\end{array}\right).\\
\end{array}
$$
It follows that $\det h_{T}(w)=\frac{2}{(1-\vert w\vert^{2})^{4}}.$ Now by Lemma \ref{lem}, we know that
$\text{trace }{K}_{T}(w)={K}_{\text{det }{ \mathcal{E}}_{T}}(w)=-\frac{4}{(1-\vert w\vert^{2})^{2}}.$
On the other hand,
$\text{trace }{K}_{\widetilde{T}} (w)={K}_{T_{1}}(w)+{K}_{T_{2}}(w)=-\frac{4}{(1-\vert w\vert^{2})^{2}},$
so that $\text{trace }{K}_{T}=\text{trace }{K}_{\widetilde{T}}.$

Now, since $T_{1,2}\notin \text{ ran }\sigma_{T_{1},T_{2}}$, it follows from Lemma \ref{lem2.10} that $T$ is strongly irreducible. Thus, $\{T\}'$ has no non-trivial idempotent elements. If $T$ and $\widetilde{T}$ were indeed similar, then there has to be an invertible operator $X\in \mathcal{L}(\mathcal{H})$ such that $T=X^{-1}\widetilde{T}X$. If we set $Y:=\Big (\begin{smallmatrix} I_{\mathcal{H}_{K_{1}}} & 0\\ 0\,\, & 0\end{smallmatrix}\Big ),$
then $Y\in \{\widetilde{T}\}'$ and $X^{-1}YX$ is a non-trivial idempotent. However,
$$(X^{-1}YX)T=X^{-1}Y\widetilde{T}X=X^{-1}\widetilde{T}YX=(X^{-1}\widetilde{T}X)(X^{-1}YX)=T(X^{-1}YX),$$
so that $X^{-1}YX\in \{T\}'$, and this is a contradiction. Hence, $T$ and $\widetilde{T}$ are not similar.
\end{example}
\section{Statements and Declarations}
\begin{enumerate}
    \item Funding: The research leading to these results received funding from Hebei Normal University under Grant Agreement No. CXZZSS2019061 and from the National Science Foundation of China under Grant Agreement No. 11922108.

    \item Conflicts of Interest/Competing Interests: The authors have no relevant financial or non-financial interests to disclose.
\end{enumerate}

\end{document}